\documentclass[a4paper,11pt,reqno]{amsart}
\usepackage{amsmath,amsfonts,amssymb,amsthm,graphicx}

\numberwithin{equation}{section}

\def\de{\delta}
\def\eps{\varepsilon}

\def\fc{\mathcal{F}}
\def\F{\mathcal{F}}
\def\hc{\mathcal{H}}
\def\H{\mathcal{H}}
\def\la{\lambda}
\def\mh{\widehat{m}}
\def\om{\overline{m}}

\def\N{\mathbb{N}}

\def\rbb{\mathbb{R}}
\def\R{\mathbb{R}}
\def\rc{\mathcal{R}}

\def\ut{\widetilde{u}}

\def\Om{\Omega}
\def\Omt{\widetilde{\Om}}
\def\Omh{\widehat{\Om}}

\def\bal{\begin{aligned}}
\def\eal{\end{aligned}}
\def\case#1#2{\par\noindent{\underline{\it Case~#1.}}\emph{ #2}\\}

\newcounter{mt}
\def\maintheorem#1#2#3{\par \smallskip \noindent {\bf Theorem \mref{#1}}~(#2){\bf .}~{\it #3}\par}
\def\mref#1{\Alph{#1}}
\def\maintheoremdeclaration#1{\stepcounter{mt}\newcounter{#1}\setcounter{#1}{\arabic{mt}}}

\maintheoremdeclaration{main}
\maintheoremdeclaration{bounded}

\newtheorem{theorem}{Theorem}[section]
\newtheorem{lemma}[theorem]{Lemma}
\newtheorem{defin}[theorem]{Definition}
\newtheorem{prop}[theorem]{Proposition}

\begin{document}
\title{Existence of minimizers for spectral problems}
\author{Dario Mazzoleni}
\address{Dipartimento di Matematica, Universit\`a degli Studi di Pavia\\
Via Ferrata,1\\ 27100 Pavia (Italy)}
\email{dario.mazzoleni@unipv.it}
\author{Aldo Pratelli}
\address{Dipartimento di Matematica, Universit\`a degli Studi di Pavia\\
Via Ferrata,1\\ 27100 Pavia (Italy)}
\email{aldo.pratelli@unipv.it}

\begin{abstract}
In this paper we show that any increasing functional of the first $k$ eigenvalues of the Dirichlet Laplacian admits a (quasi-)open minimizer among the subsets of $\R^N$ of unit measure. In particular, there exists such a minimizer which is bounded, where the bound depends on $k$ and $N$, but not on the functional. In the meantime, we show that the ratio $\lambda_k(\Omega)/\lambda_1(\Omega)$ is uniformly bounded for sets $\Omega\in\R^N$.
\end{abstract}

\maketitle

\tableofcontents

\section{Introduction}

This paper deals with the existence of minimizers for an increasing functional of the first $k$ eigenvalues of the Dirichlet Laplacian, among (quasi-)open sets in $\rbb^N$ with unit measure.\par
This kind of spectral minimization problems is well-studied and natural in many situations: for instance, one can be interested in minimizing a certain eigenvalue, or a linear combination of eigenvalues, or a product of eigenvalues, and so on. However, despite the big interest on the question, only very little is known in the general situation when the problem is stated in the whole $\R^N$. The reason is basically the lack of compactness for generic sequences of open sets; in fact, given a sequence of sets of unit measure, it is not easy to understand whether or not it converges to a limit set in a suitable sense. The first existence result (\cite[Theorem~2.5]{BM}, see also~\cite[Th\'eor\`eme~4.7.6]{HP}), which is now classical, has been found by Buttazzo and Dal Maso in the context of bounded sets.
\begin{theorem}[Buttazzo--Dal Maso, 1993]\label{BDM}
Let $D\subseteq \R^N$ be a bounded open set, and let $\{A_n\}\subseteq D$ be a sequence of (quasi-)open sets of unit measure. Then, there exists a subsequence $\{A_{n(m)}\}$ and a (quasi-)open set $A\subseteq D$ of unit measure such that
\[
\lambda_i(A) \leq \liminf_{m\to\infty} \lambda_i (A_{n(m)})
\]
for every $i\in\N$. As an immediate consequence, for every l.s.c. functional $\F:\R^k\to\R$ increasing in each variable, there exists $A$ minimizing the value of $\F\big(\lambda_1(A),\, \lambda_2(A),\, \dots \,,\, \lambda_k(A)\big)$ among the (quasi-)open sets of unit volume contained in $D$.
\end{theorem}
Thanks to a concentration-compactness argument (see for instance~\cite[Theorem~5.3.1]{H}), it is now well-known that the boundedness assumption ($A_n \subseteq D$ for every $n$) in the theorem above can be replaced by the following slightly weaker assumption: for every $\eps>0$ there exists $R=R(\eps)>0$ such that every set $A_n$ is contained in a cube of side $R$ up to a volume at most $\eps$.\par

The results described above make it simple to study spectral minimization problems inside an ambient space which is essentially bounded. On the other hand, very little is known concerning general problems in $\R^N$, because minimizing sequences, in principle, could have a significant portion of volume moving at infinity. More precisely, the existence of a minimizer for $\lambda_1$ and $\lambda_2$ is clear, since it is well-known that a ball of unit volume minimizes $\lambda_1$, while two disjoint balls of half volume minimize $\lambda_2$: these two classical results are usually referred to as Faber--Krahn inequality and Krahn--Szeg\"o inequality respectively. Instead, the existence of a minimizer for $\lambda_3$ has been proved only in 2000 by Bucur and Henrot (see~\cite{BH}), but it is still presently not known which set is the minimizer (and this is a major open problem). For other open problems and partial results see~\cite{Hart,MiMe}.\par

In this paper, we show that every increasing functional of the first $k$ eigenvalues admits a minimizer, which is in turn a bounded set. In other words, we prove that the Buttazzo--Dal Maso Theorem~\ref{BDM} is true in the whole $\R^N$, with no need of the \emph{a priori} boundedness assumption.

\maintheorem{main}{Existence of bounded minimizers}
{Let $k\in\N$, and let $\fc\colon \R^k\rightarrow\R$ be a l.s.c functional, increasing in each variable. Then there exists a bounded minimizer for the problem
\begin{equation}\label{problema}
\inf\Big\{\fc\big(\lambda_1(A),\, \lambda_2(A),\, \dots \,,\, \lambda_k(A)\big): A\subseteq \R^N,\, |A|=1 \Big\}
\end{equation}
among the (quasi-)open sets. More precisely, a minimizer $A$ is contained in a cube of side $R$, where $R$ depends on $k$ and on $N$, but \emph{not} on the particular functional $\F$.}

The strategy of the proof consists in taking a generic open set $\Omega\subseteq\R^N$ of unit volume, and showing that there exists a modified open set $\Omh$ which is uniformly bounded and which has all the first $k$ eigenvalues lower than those of $\Omega$. Roughly speaking, the basic idea why this works is that, if a set of unit volume has huge diameter, then there must be some very thin sections. And in turn, this is against the smallness of the Rayleigh quotients of the eigenfunctions, since by definition they vanish on the boundary. As an immediate consequence of the above boundedness claim, there will clearly exist a uniformly bounded minimizing sequence, and therefore Theorem~\mref{main} will follow by Theorem~\ref{BDM}.\par
During our construction, we will need the following result, simple but of indipendent interest.
\maintheorem{bounded}{Boundedness of the ratio $\lambda_k/\lambda_1$}{
There exists a constant $M=M(k,N)$ such that for every quasi-open set $A$ one has $\lambda_k(A) \leq M \lambda_1(A)$.}
It is interesting to note that, in the particular case when $k=2$, Theorem~\mref{bounded} is a consequence of the stronger result by Ashbaugh and Benguria~\cite{AB}, which states that not only the ratio $\lambda_2/\lambda_1$ is bounded, but also that the balls maximize this ratio. Notice that one cannot apply Theorem~\mref{main} to study this problem, since the functional $\lambda_k/\lambda_1$ is not increasing in the first variable. Notice also that Theorem~\mref{bounded} is stated for sets of any volume; in fact, since by trivial rescaling one finds that for every $\alpha>0$ and every $i\in\N$ one has
\begin{equation}\label{rescaling}
\lambda_i(\alpha A) = \alpha^{-2} \lambda_i(A)\,,
\end{equation}
the ratio $\lambda_k/\lambda_1$ does not change by rescaling.\par

The plan of the paper is the following. In Section~\ref{ssec2} below we give a brief list of the notation that will be used through the paper, and some preliminaries. Then, in Section~\ref{sec2} we show Theorem~\mref{main}, while in Section~\ref{sec3} we show Theorem~\mref{bounded}.\par\smallskip

\emph{{\bf \emph{NOTE:}} Dorin Bucur recently announced the article in preparation~\cite{Bucur}, where he gives a proof of the existence of a minimizer for the $k$-th eigenvalue of the Dirichlet Laplacian in $\R^N$ (which corresponds, in the language of our Theorem~\mref{main}, to the particular case when $\F$ is the projection on the last variable).}

\subsection{Notations and preliminaries\label{ssec2}}

Throughout this paper, the ambient space will be $\R^N$ for some given $N\geq 2$. The generic point of $\R^N$ will be denoted by $z\equiv (x,y)\in \R\times\R^{N-1}$, or sometimes as $z\equiv (z_1,\, z_2,\, \dots\, ,\, z_N)$, while the generic open set will be $\Omega\subseteq \R^N$. The eigenvalues of the Dirichlet Laplacian on $\Omega$ will be denoted by $\lambda_i,\, i\in\N$, while $\{u_i\}\subseteq W^{1,2}_0(\Omega)$ will be a corresponding sequence of orthogonal eigenfunctions, always normalized to have unit $L^2$ norm. For any function $v:\Omega\to\R$ and any set $D\subseteq \Omega$, we will consider the Rayleigh quotient
\[
\rc(v, D) :=  \frac{\int_D |Dv|^2}{\int_D v^2}\,;
\]
hence, in particular $\rc(u_i,\Omega)=\lambda_i(\Omega)$ for every $i\in\N$. The following is a very useful characterization of the eigenvalues, which we will need later, and whose proof can be found for instance in~\cite{H}. It is often referred to as \emph{min-max principle}.
\begin{theorem}\label{chareigen}
Let $\Omega\subseteq \R^N$ be an open set. Then for every $j\in\N$ one has
\[
\lambda_j = \min \bigg\{ \max\Big\{ \rc(w, \Omega),\, w\in K_j\setminus\{0\}\Big\}\bigg\}\,,
\]
where the minimum is taken among all the $j-$dimensional subspaces $K_j$ of $W^{1,2}_0(\Omega)$.
\end{theorem}
In our result, we need to use quasi-open sets. However, the construction of this paper only deals with open sets, and actually no knowledge of quasi-open sets is needed; therefore, we only recall here the very basic definition, addressing the interested reader for instance to~\cite{BB,HP} for a complete tractation.
\begin{defin}
Let $\Omega$ be an open set, and $V\subset\subset \Omega$ a compactly supported subset. The \emph{capacity of $V$ in $\Omega$} is defined as
\[
{\rm cap}_\Omega(V) := \inf\bigg\{ \int_\Omega |D\varphi|^2 :\, \varphi \in {\rm C}^\infty_0(\Omega),\, \varphi\geq 1 \ {\rm on}\ V\bigg\}\,.
\]
Let then $A\subseteq \R^N$ be a bounded set, and let $\Omega$ be an open set such that $A\subset\subset \Omega$. The set $A$ is said \emph{quasi-open} if for every $\eps>0$ there exists an open set $A\subseteq \Omega_\eps \subset\subset \Omega$ such that ${\rm cap}_\Omega ( \Omega_\eps \setminus A ) < \eps$. This definition does \emph{not} depend on the choice of $\Omega$. Finally, a generic set $A\subseteq \R^N$ is said \emph{quasi-open} if the intersection of $A$ with any ball of $\R^N$ is a quasi-open bounded set.
\end{defin}
It is well-known that the notion of eigenvalues can be extended to the realm of quasi-open sets. An important feature of the quasi-open sets is given by the Theorem~\ref{BDM} by Buttazzo and Dal Maso; namely, that any bounded sequence of open (or quasi-open) sets converges --up to a subsequence-- to a quasi-open set, in the sense that all the eigenvalues converge.\par

\begin{figure}[htbp]
\begin{center}
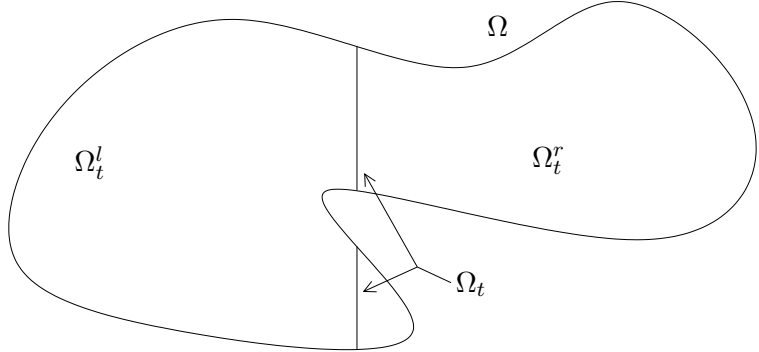
\caption{A set $\Omega$ and the corresponding sets $\Omega^l_t,\, \Omega^r_t$ and $\Omega_t$.}
\label{leftright}
\end{center}
\end{figure}

Through the paper, $k\in\N$ will be a fixed integer, and $\F:\R^k\to \R$ will be a l.s.c. function, increasing in each variable. For the sake of brevity, we will often write $\F(\Omega)$ in place of $\F\big(\lambda_1(\Omega),\, \lambda_2(\Omega),\, \dots \,,\, \lambda_k(\Omega)\big)$; hence, the goal of this paper can be rephrased by saying that we seek for a quasi-open set $A$ minimizing $\F(A)$. By density, it is clear that the infimum of $\F$ among the quasi-open sets equals the infimum among the open (or even smooth) sets.\par

The letter $C$ will be always used to denote a big geometric constant, possibly increasing from line to line; the constant $C$ will always depend \emph{only} on $N$ and on $k$ (sometimes, possibly also on some constant $K$, which in turn will eventually be chosen only depending on $N$ and $k$), thus \emph{not} on the particular choice of $\F$, and \emph{not} on the set $\Omega$. Sometimes, we will label the constants in our results as $C_1,\, C_2,\, C_3\dots$ for successive reference.\par
For any $t\in \R$, we will define
\begin{align*}
\Om^l_t:=\Big\{(x,y)\in\Om : x<t\Big\} \,, &&
\Om_t:=\Big\{y\in\rbb^{N-1} : (t,y)\in\Om\Big\}\,, &&
\Om^r_t:=\Big\{(x,y)\in\Om : x>t\Big\}\,;
\end{align*}
notice that $\Omega_t^l$ and $\Omega_t^r$ are subsets of $\R^N$, while $\Omega_t$ is a subset of $\R^{N-1}$. Figure~\ref{leftright} shows an example of a generic set $\Omega$ with $\Omega_t^l,\, \Omega_t^r$ and $\Omega_t$. On the other hand, given $0\leq m \leq |\Om|$ and $0\leq m_1 \leq m_2 \leq |\Omega|$, we define the level $\tau(\Omega,m)\in \overline \R$ and the width $W(\Omega,m_1,m_2)$ as
\begin{align*}
\tau(\Omega,m):= \inf\Big\{ t \in \R:\, \big| \Omega^l_t\big| \geq m \Big\}\,, &&
W(\Omega,m_1,m_2):= \tau(\Omega,m_2)-\tau(\Omega,m_1)\,.
\end{align*}
Observe that one surely has $-\infty<\tau(\Omega,m)<+\infty$ whenever $0<m<|\Omega|$, as well as $W(\Omega,m_1,m_2)<+\infty$ if $0<m_1\leq m_2 < |\Omega|$.\par
Finally, given any set $\Omega\subseteq\R^N$, we define its $1$-dimensional projections for $1\leq p \leq N$ as
\[
\pi_p(\Omega) := \Big\{ t\in\R:\, \exists \, (z_1,\, z_2,\, \dots\, ,\, z_N)\in\Omega,\, z_p = t\Big\}\,.
\]

\section{Proof of Theorem~\mref{main}}\label{sec2}

In this section we present the proof of Theorem~\mref{main}. As already anticipated in the introduction, our strategy basically consists in showing that to minimize $\F$ it is enough to concentrate on uniformly bounded sets. More precisely, we will show the following result.

\begin{prop}\label{allhere}
For every $K>0$ there exists a constant $R=R(k,K,N)$, such that the following holds. If $\Omega\subseteq \R^N$ is an open set of unit volume and with $\lambda_k(\Omega)\leq K$, there exists another open set $\Omh$, still of unit volume but contained in a cube of side $R$, and with $\lambda_i(\Omh)\leq \lambda_i(\Om)$ for every $1\leq i \leq k$.
\end{prop}

Let us immediately see how Theorem~\mref{main} follows from this proposition, then the rest of the section will be devoted to show the proposition.

\begin{proof}[Proof of Theorem~\mref{main}]
Let us take a minimizing sequence of open sets $\{\Omega_n\}$ for problem~(\ref{problema}). Fix a generic $n\in\N$, and assume for a moment that $\lambda_k(\Omega_n)\geq M\lambda_k(B)$, being $B$ the ball of unit volume and $M$ the constant of Theorem~\mref{bounded}. If it is so, then by Theorem~\mref{bounded} one has $\lambda_1(\Omega_n)\geq \lambda_k(B)$, thus for every $1\leq i \leq k$ it is $\lambda_i(\Omega_n)\geq \lambda_i(B)$, hence $\F(\Omega_n)\geq\F(B)$, being $\F$ increasing in each variable. Thanks to this observation, it is admissible to assume that $\lambda_k(\Omega_n)\leq K := M \lambda_k(B)$ for every $n$. By Proposition~\ref{allhere}, then, there exists another sequence $\{\Omh_n\}$, made by open sets of unit volume contained in a cube of side $R$, with $\lambda_i(\Omh_n)\leq \lambda_i(\Om_n)$ for every $1\leq i \leq k$ and every $n\in\N$. Again by the assumption that $\F$ is increasing in each variable, we derive that also $\{\Omh_n\}$ is a minimizing sequence for~(\ref{problema}).\par
We can then apply Theorem~\ref{BDM} to find a quasi-open set $A$, still contained in the cube of side $R$, and such that, up to extract a subsequence of $\{\Omh_n\}$, one has $\la_i(A)\leq \liminf_n \lambda_i(\Omh_n)$ for every $1\leq i \leq k$. By the lower semi-continuity of $\F$, we derive that $A$ is a minimizer for $\F$, thus the proof is concluded.
\end{proof}

The rest of this section is devoted to show Proposition~\ref{allhere}. For the ease of presentation, we divide the construction in three subsections. In the first one we obtain the boundedness of the ``tails'' (Lemma~\ref{lemmatail}), while in the second one we consider the internal part (Lemma~\ref{lemmainterior}). Then, in the last subsection we put everything together to give the proof of Proposition~\ref{allhere}.

\subsection{Boundedness of the tails\label{sub1}}

This subsection is devoted to show that, under the assumptions of Proposition~\ref{allhere}, we can reduce to the case when the ``tails'' of $\Omega$ are bounded. More precisely, we fix once for all a small positive number $\mh=\mh(K,N)\in (0,1/4)$ in such a way that
\begin{equation}\label{defmh}
\frac{(4\mh)^{\frac 2N}}{\lambda_1(B_N)} \, K \leq \frac 12\,,
\end{equation}
being $B_N$ the ball of unit volume in $\R^N$. We aim to show the following result.

\begin{lemma}\label{lemmatail}
For every $K>0$ there exist $R_1=R_1(k,K,N)$ and $\Gamma_1=\Gamma_1(k,K,N)$ such that, for any open set $\Omega\subseteq \R^N$ of unit volume and with $\lambda_k(\Omega)\leq K$, there exists another open set $\Omh\subseteq\R^N$, still of unit volume, such that $\lambda_i(\Omh)\leq \lambda_i(\Om)$ for every $1\leq i \leq k$, and that for every $2\leq p \leq N$
\begin{gather}
W\big(\Omh,0,\mh\big) \leq R_1\,, \label{estimatetail} \\
W\big(\Omh,\mh,1\big) \leq \Gamma_1 \Big( W\big(\Om,\mh,1\big)\Big)\,,
\qquad {\rm diam} \big(\pi_p(\Omh)\big)\leq \Gamma_1\, {\rm diam} \big(\pi_p(\Omega)\big)\,. \label{estimatediam}
\end{gather}
\end{lemma}

The claim of the lemma, roughly speaking, says that it is always possible to assume that the ``tail'' of $\Omega$, i.e., the set $\Omega^l_{\tau(\Omega,\mh)}$ of volume $\mh$, has horizontal projection of length at most $R_1$. More precisely, condition~(\ref{estimatetail}) says that one can modify $\Omega$ in such a way that the tail is uniformly horizontally bounded, while condition~(\ref{estimatediam}) says that this modification does not excessively worsen the remaining part of the set $\Omega$, nor its extension in the $N-1$ non-horizontal directions.\par

To prove the lemma, we start setting for brevity $\bar t =\tau(\Omega,2\mh)$, and for every $t\leq \bar t$ we define
\begin{align}\label{int0}
\Omega^+(t) := \Omega^r_t\,, && \Omega^-(t) := \Omega^l_t\,, && \eps(t):=\hc^{N-1}(\Om_t)\,.
\end{align}
Observe that
\begin{equation}\label{int1}
m(t) := \big| \Omega^-(t) \big| = \int_{-\infty}^t \eps(s)\,ds\leq 2\mh\,.
\end{equation}
Moreover, we let as usual $\big\{u_1,\, u_2 ,\, \dots \,,\, u_k\big\}$ be an orthonormal set of eigenfunctions with unit $L^2$ norm and corresponding to the first $k$ eigenvalues of $\Omega$. We define then also, for every $1\leq i \leq k$ and every $t\leq \bar t$,
\begin{align}\label{int2}
\de_i(t):=\int_{\Om_t}{|Du_i(t,y)|^2\,d\hc^{N-1}(y)}\,, && \mu_i(t):=\int_{\Om_t}{u_i(t,y)^2\,d\hc^{N-1}(y)}\,,
\end{align}
which makes sense since every $u_i$ is smooth. It is convenient to give the further notation
\[
\delta(t) := \sum_{i=1}^k \delta_i(t) = \sum_{i=1}^k \int_{\Omega_t} \big|Du_i(t,y)\big|^2\, d\H^{N-1}(y)\,,
\]
and in analogy with~(\ref{int1}) we also set
\begin{equation}\label{defphi}
\phi(t) := \sum_{i=1}^k \int_{\Omega^-(t)}  |Du_i|^2 = \int_{-\infty}^t \delta(s)\,ds \,.
\end{equation}
Applying the Faber--Krahn inequality in $\R^{N-1}$ to the set $\Omega_t$, and using~(\ref{rescaling}) on $\R^{N-1}$, we know that
\[
\eps(t)^{\frac{2}{N-1}} \lambda_1(\Omega_t)=\H^{N-1}(\Omega_t)^{\frac{2}{N-1}} \lambda_1(\Omega_t) \geq \lambda_1(B_{N-1})\,,
\]
calling $B_{N-1}$ the unit ball in $\R^{N-1}$. As a trivial consequence, we can estimate $\mu_i$ in terms of $\eps$ and $\de_i$: in fact, noticing that $u_i(t,\cdot)\in W^{1,2}_0(\Omega_t)$ and writing $Du_i = (D_1 u_i, D_y u_i)$, we have
\begin{equation}\label{muest}
\mu_i(t)=\int_{\Om_t}{u_i(t,\cdot)^2\,d\hc^{N-1}}\leq \frac{1}{\la_1(\Om_t)}\int_{\Om_t}{|D_y u_i(t,\cdot)|^2\,d\hc^{N-1}}\leq C\eps(t)^{\frac{2}{N-1}}\de_i(t).
\end{equation}
We can now present two estimates which assure that $u_i$ and $Du_i$ can not be too big in $\Om^-(t)$.

\begin{lemma}\label{primastima}
Under the assumptions of Lemma~\ref{lemmatail}, for every $1\leq i\leq k$ and $t\leq \bar t$ the following inequalities hold: 
\begin{align}\label{udu-}
\int_{\Om^-(t)} u_i^2 \leq C_1 \eps(t)^{\frac{1}{N-1}}\de_i(t)\,, &&
\int_{\Om^-(t)} |Du_i|^2 \leq C_1 \eps(t)^{\frac{1}{N-1}}\de_i(t)\,,
\end{align}
for some $C_1=C_1(k,K,N)$.
\end{lemma}
\begin{proof}
Let us fix $t\leq \bar t$. Consider the set $\Om_S^-$ obtained by the union of $\Om^-(t)$ and its reflection with respect to the plane $\{x=t\}$, and call $u_S \in W^{1,2}_0(\Om_S)$ the function obtained by reflecting $u_i$. Calling $B_N$ the unit ball in $\R^N$, we find then
\[
\frac{\lambda_1(B_N)}{\big(2m(t)\big)^{\frac{2}{N}}} = \frac{\lambda_1(B_N)}{|\Om_S^-|^{\frac{2}{N}}}
\leq \la_1(\Om_S^-) \leq \rc(u_S,\Om^-_S)= \rc\big(u_i,\Om^-(t)\big)
= \frac{\bal \int_{\Om^-(t)} |Du_i|^2\eal}{\bal \int_{\Om^-(t)} u_i^2 \eal}\,,
\]
by the symmetry of $\Om^-_S$, and using again~(\ref{rescaling}). This estimate gives
\begin{equation}\label{eq1}
\int_{\Om^-(t)} u_i^2 \leq \frac{\big(2m(t)\big)^{\frac 2N}}{\lambda_1(B_N)}\, \int_{\Om^-(t)} |Du_i|^2
\end{equation}
which in particular, being $m(t)\leq 2\mh$ and recalling~(\ref{defmh}), implies
\begin{equation}\label{estray2}
\int_{\Omega^-(t)} u_i^2 \leq \frac 12\,.
\end{equation}
On the other hand, recalling that $-\Delta u_i = \lambda_i u_i$, by Schwarz inequality and using~(\ref{muest}) we have
\begin{equation}\label{eq2}\begin{split}
\int_{\Om^-(t)} |Du_i|^2 &=  \int_{\Om^-(t)} \lambda_i u_i^2 + \int_{\Om_t}  u_i \,\frac{\partial u_i}{\partial \nu}
\leq K \int_{\Om^-(t)} u_i^2 + \sqrt{\int_{\Om_t} u_i^2  \int_{\Om_t} |Du_i|^2}\\
&\leq K \int_{\Om^-(t)} u_i^2 + C \eps(t)^{\frac 1{N-1}} \delta_i(t)\,.
\end{split}\end{equation}
It is now easy to obtain~(\ref{udu-}) combining~(\ref{eq1}) and~(\ref{eq2}). In fact, by inserting the latter into the first, we find
\[
\int_{\Om^-(t)} u_i^2 \leq  \frac{\big(2m(t)\big)^{\frac 2N}}{\lambda_1(B_N)}\, \bigg( K \int_{\Om^-(t)} u_i^2 + C \eps(t)^{\frac 1{N-1}} \delta_i(t) \bigg)\,,
\]
which by~(\ref{defmh}) again yields
\begin{equation}\label{leftudu-}
\frac{1}{2} \int_{\Om^-(t)} u_i^2 \leq  \frac{\big(2m(t)\big)^{\frac 2N}}{\lambda_1(B_N)}\, C \eps(t)^{\frac 1{N-1}} \delta_i(t) 
\leq C \eps(t)^{\frac 1{N-1}} \delta_i(t) \,.
\end{equation}
The left estimate in~(\ref{udu-}) is then obtained. To get the right one, one has then just to insert~(\ref{leftudu-}) into~(\ref{eq2}).
\end{proof}

Let us go further into our construction, giving some definitions. For any $t\leq \bar t$ and $\sigma(t)>0$, we define the cylinder $Q(t)$, shown in Figure~\ref{taglio}, as
\begin{equation}\label{defcyl}
Q(t):=\Big\{(x,y)\in \rbb^N: \, t-\sigma < x < t,\ (t,y) \in\Om\Big\} = \big(t-\sigma,t\big) \times \Omega_t\,,
\end{equation}
where for any $t\leq \bar t$ we set
\begin{equation}\label{defsigma}
\sigma(t)= \eps(t)^{\frac 1{N-1}}\,.
\end{equation}
We let also $\Omt(t)=\Omega^+(t)\cup Q(t)$, and we introduce $\ut_i\in W^{1,2}_0\big(\Omt(t)\big)$ as
\begin{equation}\label{utilde}
\ut_i(x,y):=\left\{
\begin{array}{ll}
u_i(x,y) &\hbox{if $(x,y)\in \Omega^+(t)$}\,, \\[5pt]
\bal\frac{x-t+\sigma}{\sigma}\,u_i(t,y)\eal &\hbox{if $(x,y)\in Q(t)$}\,.
\end{array}
\right.
\end{equation}
The fact that $\ut_i$ vanishes on $\partial\Omt(t)$ is obvious; moreover, $Du_i=D\ut_i$ on $\Omega^+(t)$, while on $Q(t)$ one has
\begin{equation}\label{estdut}
D\ut_i(x,y)=\left(\frac{u_i(t,y)}{\sigma}\,,\,\frac{x-t+\sigma}{\sigma}\,D_y u_i(t,y)\right)\,.
\end{equation}
\begin{figure}[htbp]
\begin{center}
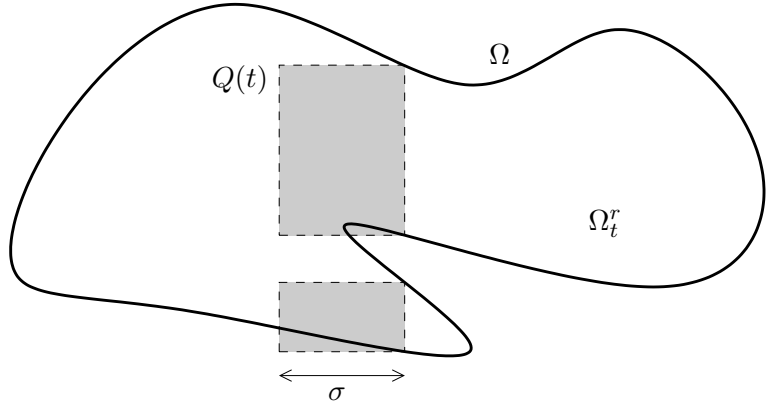
\caption{A set $\Omega$ with the cylinder $Q(t)$ (shaded).}
\label{taglio}
\end{center}
\end{figure}
A simple calculation allows us to estimate the integrals of $\ut_i$ and $D\ut_i$ on $Q(t)$.
\begin{lemma}\label{lemmatest}
For every $t\leq \bar t$ and $1\leq i \leq k$, one has
\begin{align}\label{newtest}
\int_{Q(t)}|D\ut_i|^2 \leq C_2\eps(t)^{\frac{1}{N-1}}\de_i(t) \,, &&
\int_{Q(t)} \ut_i^2 \leq C_2 \eps(t)^{\frac 3{N-1}} \delta_i(t)\,,
\end{align}
for a suitable constant $C_2=C_2(k,K,N)$.
\end{lemma}
\begin{proof}
Thanks to~(\ref{estdut}), and using also~(\ref{muest}) and~(\ref{defsigma}), one obtains the first estimate in~(\ref{newtest}) since
\[\begin{split}
\int_{Q(t)} |D\ut_i(x,y)|^2 \,dx\,dy
&= \int_{t-\sigma}^t \int_{\Omega_t} \frac{u_i^2( t,y)}{\sigma^2} + \frac{(x-t+\sigma)^2}{\sigma^2}\,|D_y u_i(t,y)|^2\,dy \,dx
= \frac {\mu_i}\sigma +\frac{\delta_i\sigma}3\\
&\leq C\,\frac{\eps(t)^{\frac{2}{N-1}}\de_i(t)}{\sigma} + \frac{\delta_i \sigma}3
= C\eps(t)^{\frac{1}{N-1}}\de_i(t) \,.
\end{split}\]
On the other hand, the second estimate in~(\ref{newtest}) follows, also again using~(\ref{muest}) and~(\ref{defsigma}), by
\[
\int_{Q(t)} \ut_i(x,y)^2\,dx\,dy = \int_{t-\sigma}^t \int_{\Omega_t}  \ut_i^2(t,y)\,dy \,dx = \frac{\sigma\mu_i} 3
\leq C \sigma\eps(t)^{\frac 2{N-1}} \delta_i(t)=C \eps(t)^{\frac 3{N-1}} \delta_i(t)\,.
\]
\end{proof}

Another simple but useful estimate concerns the Rayleigh quotients of the functions $\ut_i$ on the sets $\Omt(t)$ and the integral of the products $\ut_i \ut_j$.

\begin{lemma}\label{noth}
For every $t\leq \bar t$ and $1\leq i \leq k$, one has
\begin{equation}\label{est1a}
\rc\big(\ut_i,\Omt(t)\big)\leq \la_i(\Om)+ C\eps(t)^{\frac 1 {N-1}}\de_i(t)\,.
\end{equation}
Moreover, for every $i\neq j\in \{1,\, 2,\, \dots \,,\, k\}$, one has
\begin{equation}\label{est1b}
\bigg|\int_{\Omt(t)}  \ut_i \ut_j + D\ut_i \cdot D\ut_j\bigg| \leq C\Big(\eps(t)^{\frac 3{N-1}}+\eps(t)^{\frac 1{N-1}}\Big)\sqrt{\de_i(t)\de_j(t)}\,.
\end{equation}
\end{lemma}
\begin{proof}
Recalling that $-\Delta u_i = \lambda_i(\Omega) u_i$, making use of~(\ref{estray2}) and~(\ref{newtest}) and arguing as in~(\ref{eq2}), we obtain
\[\begin{split}
\rc\big(\ut_i,\Omt(t)\big)
&=\frac{\bal\int_{\Om^+(t)} |D\ut_i|^2+\int_{Q(t)} |D\ut_i|^2\eal}{\bal\int_{\Om^+(t)} \ut_i^2 + \int_{Q(t)} \ut_i^2\eal}
\leq \frac{\bal\int_{\Om^+(t)} |Du_i|^2+\int_{Q(t)} |D\ut_i|^2\eal}{\bal\int_{\Om^+(t)} u_i^2\eal}\\
&= \frac{\bal\lambda_i(\Omega) \int_{\Om^+(t)} u_i^2+\int_{\Om_t} u_i \frac{\partial u_i}{\partial \nu}+ \int_{Q(t)} |D\ut_i|^2\eal}{\bal\int_{\Om^+(t)} u_i^2\eal}
\leq \lambda_i(\Omega) + C\eps(t)^{\frac 1{N-1}}\de_i(t)\,,
\end{split}\]
hence~(\ref{est1a}) is proved.\par
On the other hand, recall that $u_i$ and $u_j$ are orthogonal on $\Om$ both in $L^2$ and in $W^{1,2}_0$ sense by definition, hence by using~(\ref{udu-}) and~(\ref{newtest}) we find
\[\begin{split}
\bigg|\int_{\Omt(t)} \ut_i \ut_j\bigg| &\leq  \bigg|\int_{\Omega^+(t)} u_i u_j  \bigg| + \bigg| \int_{Q(t)} \ut_i \ut_j\bigg|
=  \bigg|\int_{\Omega^-(t)} u_i u_j  \bigg| + \bigg| \int_{Q(t)} \ut_i \ut_j\bigg|\\
&\leq C\bigg(\eps(t)^{\frac{1}{N-1}}\sqrt{\de_i(t)\de_j(t)}+\eps(t)^{\frac{3}{N-1}}\sqrt{\de_i(t)\de_j(t)}\bigg)\,.
\end{split}\]
In the very same way, concerning $D\ut_i$ and $D\ut_j$, we have
\[\begin{split}
\bigg|\int_{\Omt(t)} D\ut_i \cdot D\ut_j\bigg| &\leq  \bigg|\int_{\Omega^+(t)} Du_i \cdot Du_j  \bigg| + \bigg| \int_{Q(t)} D\ut_i \cdot D\ut_j\bigg|\\
&=  \bigg|\int_{\Omega^-(t)} Du_i \cdot Du_j  \bigg| + \bigg| \int_{Q(t)} D\ut_i \cdot D\ut_j\bigg|
\leq C\eps(t)^{\frac{1}{N-1}}\sqrt{\de_i(t)\de_j(t)}\,.
\end{split}\]
Adding up the last two estimates yields~(\ref{est1b}).
\end{proof}

In order to prove Lemma~\ref{lemmatail}, we need to compare the eigenvalues of $\Omega$ and those of $\Omt(t)$; this can be done by means of the min-max principle, Theorem~\ref{chareigen}, which relates the eigenvalues with the Rayleigh quotients of $W^{1,2}_0$ functions.

\begin{lemma}\label{lemmanu}
There exist a small constant $\nu=\nu(k,K,N)<1$ and a constant $C_3=C_3(k,K,N)$ such that, if $\eps(t),\, \delta_i(t) \leq \nu$ for every $1\leq i \leq k$, then
\begin{align}\label{est2new}
\la_j\big(\Omt(t)\big)\leq\la_j(\Om)+C_3 \eps(t)^{\frac 1 {N-1}}\de(t) && \forall\, 1\leq j \leq k\,.
\end{align}
\end{lemma}
\begin{proof}
We aim to use the characterization given by Theorem~\ref{chareigen}. To do so, for every $1\leq j \leq k$ we define $K_j$ as the linear subspace of $W^{1,2}_0\big(\Omt(t)\big)$ spanned by the functions $\ut_i$ with $1\leq i \leq j$. First of all, we give the\\
{\bf Claim A.} For every $1\leq j\leq k$, the space $K_j$ has dimension $j$.\\
Suppose that the claim is not true. Then, there should exist some $1 \leq \ell \leq k$ and some coefficients $\beta_i$ for $i \neq \ell$ with all $|\beta_i|\leq 1$ and
\[
\ut_{\ell} = \sum_{1 \leq i \leq k,\, i \neq \ell} \beta_i \ut_i\,.
\]
Notice now that by~(\ref{estray2}) we know that $\int_{\Omega^+(t)} u_{\ell}^2 \geq 1/2$, hence also by~(\ref{est1b}) we deduce
\[
\frac 12 \leq \int_{\Omt(t)}\ut_\ell^2 = \int_{\Omt(t)} \sum_{1 \leq i \leq k,\, i \neq \ell} \beta_i \ut_i \ut_\ell
\leq \sum_{1 \leq i \leq k,\, i \neq \ell} \bigg|\int_{\Omt(t)} \ut_i \ut_\ell\bigg|
\leq k\, C\Big( \nu^{\frac 3{N-1}} +  \nu^{\frac 1{N-1}}\Big) \nu < \frac 12\,,
\]
where the last inequality is true provided that $\nu=\nu(k,K,N)$ is chosen small enough. The absurd shows the validity of Claim~A.\par
We can now show~(\ref{est2new}): to do so, pick a generic function $w\in K_j$, which can be written (up to a rescaling) as $w = \sum_{i=1}^j \beta_i \ut_i$, where $\max \big\{|\beta_i|,\, 1\leq i\leq j\big\}= 1$. We need to evaluate $\rc\big(w,\Omt(t)\big)$: we start by noticing that
\begin{equation}\label{jts}\begin{split}
\rc\big(w,\Omt(t)\big) &= \frac{\int_{\Omt(t)} |Dw|^2}{\int_{\Omt(t)} w^2}
= \frac{\sum_{i=1}^j \beta_i^2\int_{\Omt(t)} |D\ut_i|^2 + \sum_{i\neq j} \beta_i \beta_j \int_{\Omt(t)} D\ut_i \cdot D\ut_j}{\sum_{i=1}^j \beta_i^2\int_{\Omt(t)} \ut_i^2 + \sum_{i\neq j} \beta_i \beta_j \int_{\Omt(t)} \ut_i\ut_j}\\
&\leq \frac{\sum_{i=1}^j \beta_i^2\int_{\Omt(t)} |D\ut_i|^2 + C k^2 \eps(t)^{\frac 1{N-1}}\delta (t)}{\sum_{i=1}^j \beta_i^2\int_{\Omt(t)} \ut_i^2 - C k^2 \eps(t)^{\frac 1{N-1}}\delta (t)}\,,
\end{split}\end{equation}
where the last inequality comes by~(\ref{est1b}). If $\nu(k,K,N)$ is small enough, then
\[
Ck^2\eps(t)^{\frac 1{N-1}} \delta(t) \leq C k^2 \nu^{\frac N{N-1}}\leq \frac14\,,
\]
hence by the choice of $\beta_i$ and by~(\ref{estray2}) the denominator in the last fraction of~(\ref{jts}) is bigger than $1/4$ (in particular, it is strictly positive). As a consequence, recalling also that by~(\ref{est1a}) one has for every $1\leq i \leq j$
\[
\frac{\bal\int_{\Omt(t)} |D\ut_i|^2\eal}{\bal\int_{\Omt(t)} \ut_i^2\eal} 
\leq \lambda_i(\Omega) + C\eps(t)^{\frac 1{N-1}}\delta_i(t) 
\leq \lambda_j(\Omega) + C\eps(t)^{\frac 1{N-1}}\delta(t)\,,
\]
from~(\ref{jts}) we deduce
\[\begin{split}
\rc\big(w,\Omt(t)\big)&\leq
\frac{\Big(\lambda_j(\Omega) + C\eps(t)^{\frac 1{N-1}}\delta(t)\Big)
\bigg(\sum_{i=1}^j \beta_i^2\int_{\Omt(t)} \ut_i^2\bigg) + C \eps(t)^{\frac 1{N-1}}\delta (t)}{\sum_{i=1}^j \beta_i^2\int_{\Omt(t)} \ut_i^2 - C  \eps(t)^{\frac 1{N-1}}\delta (t)}\\
&\leq \frac{\Big(\lambda_j(\Omega) + C\eps(t)^{\frac 1{N-1}}\delta(t)\Big)+ 2C \eps(t)^{\frac 1{N-1}}\delta (t)}{1- 2C \eps(t)^{\frac 1{N-1}}\delta(t)}
\leq \lambda_j(\Omega) + C \eps(t)^{\frac 1{N-1}}\delta(t)
\end{split}\]
(keep in mind that the constant $C=C(k,K,N)$ may increase from line to line). The validity of~(\ref{est2new}) is then now an immediate consequence of Theorem~\ref{chareigen} and of Claim~A.
\end{proof}

We can now enter in the central part of our construction. Basically, we aim to show that either $\Omega$ already satisfies the requirements of Lemma~\ref{lemmatail}, or some $\Omt(t)$ does it, up to a rescaling. To do so, we need another definition, namely, for every $t\leq \bar t$ we define the rescaled set
\[
\Omh(t) := \big| \Omt(t) \big|^{-\frac 1N} \Omt(t)\,,
\]
so that $\big| \Omh(t)\big|=1$. We can now show the following result.

\begin{lemma}\label{threeconditions}
Let $\Omega$ be as in the assumptions of Lemma~\ref{lemmatail}, and let $t\leq \bar t$. There exists $C_4=C_4(k,K,N)$ such that exactly one of the three following conditions hold:
\begin{enumerate}
\item[(1) ] $\max\big\{ \eps(t),\, \delta(t) \big\} > \nu$;
\item[(2) ] (1) does not hold and $m(t) \leq C_4 \big( \eps(t) + \delta(t)\big) \eps(t)^{\frac 1{N-1}}$;
\item[(3) ] (1) and~(2) do not hold and for every $1\leq i \leq k$, one has $\lambda_i\big(\Omh(t)\big)< \lambda_i(\Omega)$.
\end{enumerate}
In particular, if condition~(3) holds for $t$ and $m(t)\geq \mh$, then for every $1\leq i \leq k$ one has $\lambda_i\big(\Omh(t)\big) < \lambda_i(\Om) - \eta$, being $\eta=\eta(k,K,N)>0$.
\end{lemma}
\begin{proof}
If~(1) holds true, there is of course nothing to prove. Otherwise, it is possible to apply Lemma~\ref{lemmanu}, hence  we have
\begin{equation}\label{putinto}
\la_i\big(\Omt(t)\big)\leq\la_i(\Om)+C_3\eps(t)^{\frac 1 {N-1}}\delta(t)
\end{equation}
for every $1\leq i \leq k$. By~(\ref{rescaling}) and the fact that $\big| \Omh(t)\big|=1$, we know that
\[
\lambda_i\big( \Omh(t)\big)= \big| \Omt(t) \big|^{\frac 2N} \lambda_i\big(\Omt(t)\big)\,.
\]
By construction,
\[
\big| \Omt(t)\big|= \big| \Om^+(t)\big|+ \big|Q(t)\big| = 1 - m(t) + \eps(t)^{\frac N{N-1}}\,,
\]
hence the above estimates and~(\ref{putinto}) lead to
\begin{equation}\label{muchhere}\begin{split}
\lambda_i(\Omh(t)) &= \Big( 1 - m(t) + \eps(t)^{\frac N{N-1}} \Big)^{\frac 2N} \lambda_i\big(\Omt(t)\big)\\
&\leq \Big( 1 - \frac 2N \, m(t) + \frac 2N\, \eps(t)^{\frac N{N-1}} \Big)  \Big( \la_i(\Om)+C_3\eps(t)^{\frac 1 {N-1}}\delta(t) \Big)\\
&\leq \lambda_i(\Omega) - \frac{2\lambda_1(B_N)}{N}\, m(t)+\frac{2K}N\, \eps(t)^{\frac N{N-1}}+\bigg(C_3+\frac 2N\bigg)\eps(t)^{\frac 1 {N-1}}\delta(t)\,.
\end{split}\end{equation}
This allows us to conclude. In fact, defining $C_4:= \frac{2(K+1)}N+C_3$, if $m(t) \leq C_4 \big( \eps(t) + \delta(t)\big) \eps(t)^{\frac 1{N-1}}$, then condition~(2) holds true. Otherwise, (\ref{muchhere}) directly implies that $\lambda_i\big(\Omh(t)\big) < \lambda_i(\Om)$.\par
Finally, assume that condition~(3) holds and $m(t)\geq \mh$: in this case, (\ref{muchhere}) directly implies
\[
\lambda_i\big(\Omh(t)\big) - \lambda_i(\Om)\leq - \frac{2\lambda_1(B_N)}{N}\, \mh+C_4\nu^{\frac N{N-1}} \leq -\eta\,,
\]
where $\eta=\lambda_1(B_N) \mh / N$ and the last inequality is true up to decrease $\nu$ (notice that \emph{decreasing} the value of the constant $\nu$ of Lemma~\ref{lemmanu} does not change the value of $C_3$).
\end{proof}

We are finally in position to give the proof of Lemma~\ref{lemmatail}.
\begin{proof}[Proof of Lemma~\ref{lemmatail}]
Let us start defining
\begin{equation}\label{defhatt}
\hat t := \sup \Big\{ t \leq \bar t : \, \hbox{condition~(3) of Lemma~\ref{threeconditions} holds for $t$}\Big\}\,,
\end{equation}
with the usual convention that, if condition~(3) is false for every $t\leq \bar t$, then $\hat t = -\infty$. We introduce now the following subsets of $(\hat t,\bar t)$,
\begin{align*}
A:&= \Big\{ t\in (\hat t,\bar t):\, \hbox{condition~(1) of Lemma~\ref{threeconditions} holds for $t$} \Big\}\,,\\
B:&=\Big\{t\in (\hat t,\bar t):\, \hbox{condition~(2) of Lemma~\ref{threeconditions} holds for $t$ and $m(t)>0$} \Big\}\,,
\end{align*}
and we further subdivide them as
\begin{align*}
A_1:= \Big\{ t \in A:\, \eps(t)\geq \delta(t) \Big\}\,, && A_2:= \Big\{ t \in A:\, \eps(t)<\delta(t) \Big\}\,, \\
B_1:= \Big\{ t \in B:\, \eps(t)\geq \delta(t) \Big\}\,, && B_2:= \Big\{ t \in B:\, \eps(t)<\delta(t) \Big\}\,.
\end{align*}
We aim to show that both $A$ and $B$ are uniformly bounded. Concerning $A_1$, observe that
\[
\nu\big|A_1\big| \leq \int_{A_1} \eps(t)\, dt = \Big| \Big\{ (x,y)\in \Omega:\, x \in A_1\Big\}\Big| \leq \big| \Omega\big|=1\,,
\]
so that $|A_1|\leq 1/\nu$. Concerning $A_2$, in the same way and also recalling that $\lambda_i(\Omega)\leq K$ for every $i\leq k$, we have
\[
\nu\big|A_2\big| \leq \int_{A_2} \delta(t)\, dt 
= \sum_{i=1}^k \int_{A_2 }\int_{\Omega_t} \big|Du_i(t,y)\big|^2\, d\H^{N-1}(y)\,dt 
\leq \sum_{i=1}^k \int_\Omega \big|Du_i\big|^2 \leq kK\,,
\]
so that $|A_2|\leq kK/\nu$. Summarizing, we have proved that
\begin{equation}\label{estA}
\big| A \big| \leq \frac{1 + kK}\nu\,.
\end{equation}
Let us then pass to the set $B_1$. To deal with it, we need a further subdivision, namely, we write $B_1= \cup_{n\in\N} B_1^n$, where
\begin{equation}\label{defB1n}
B_1^n := \bigg\{ t\in B_1:\, \frac{\mh}{2^n} < m(t) \leq \frac{\mh}{2^{n-1}} \bigg\}\,.
\end{equation}
Keeping in mind~(\ref{int1}), we know that $t\mapsto m(t)$ is an increasing function, and that for a.e. $t \in\R$ one has $m'(t) = \eps(t)$. Moreover, for every $t\in B_1$ one has by construction that
\[
m(t)\leq C_4 \big( \eps(t) + \delta(t)\big) \eps(t)^{\frac 1{N-1}}\leq 2 C_4\, \eps(t)^{\frac N{N-1}}\,.
\]
As a consequence, for every $t\in B_1^n$ one has
\[
m'(t) = \eps(t) \geq \frac{1}{C}\, m(t)^{\frac{N-1}N}
\geq \frac{1}{C}\, \mh^{\frac{N-1}N} \, \frac{1}{\big(2^{\frac{N-1}N}\big)^n}\,.
\]
This readily implies
\[
\frac{1}{C}\, \mh^{\frac{N-1}N} \, \frac{1}{\big(2^{\frac{N-1}N}\big)^n} \, \big| B_1^n \big| 
\leq \int_{B_1^n} m'(t) \leq \frac{\mh}{2^n}\,,
\]
which in turn gives
\[
\big| B_1^n \big| \leq  C \mh^{\frac 1N} \big(2^{-\frac 1N}\big)^n\,.
\]
Finally, we deduce
\begin{equation}\label{estB1}
\big| B_1 \big| = \sum_{n\in\N} \big| B_1^n\big| 
\leq C \mh^{\frac 1N} \sum_{n\in\N} \big(2^{-\frac 1N}\big)^n
= C \mh^{\frac 1N} \,\frac{2^{\frac 1N}}{2^{\frac 1N}-1}\,.
\end{equation}
Notice that basically our argument consisted in using the fact that in $B_1$ one has
\begin{align}\label{analof}
m(t)\leq C\eps(t)^{\frac N{N-1}}\,, && \hbox{with } \eps(t) = m'(t)\,.
\end{align}
Concerning $B_2$, we can almost repeat the same argument: in fact, thanks to~(\ref{udu-}), for every $t\in B_2$ we have
\begin{align*}
\phi(t)=\sum_{i=1}^k\int_{\Om^-(t)}|Du_i|^2 \leq C_1\,\eps(t)^{\frac{1}{N-1}}\delta(t)\leq C_1\,\delta(t)^{\frac{N}{N-1}}\,, && 
\hbox{with } \delta(t) = \phi'(t)\,.
\end{align*}
which is the perfect analogous of~(\ref{analof}) with $\delta$ and $\phi$ in place of $\eps$ and $m$ respectively. Since as already observed $\phi(\bar t) \leq \sum \int_\Omega |Du_i|^2 \leq kK$, in analogy with~(\ref{defB1n}) we can define
\[
B_2^n := \bigg\{ t\in B_2:\, \frac{kK}{2^{n+1}} < \phi(t) \leq \frac{kK}{2^n} \bigg\}\,,
\]
thus the very same argument which lead to~(\ref{estB1}) now gives
\begin{equation}\label{estB2}
\big| B_2 \big| = \sum_{n\in\N} \big| B_2^n\big| 
\leq C \big(kK\big)^{\frac 1N} \sum_{n\in\N} \big(2^{-\frac 1N}\big)^n
= C \big(kK\big)^{\frac 1N} \,\frac{2^{\frac 1N}}{2^{\frac 1N}-1}\,.
\end{equation}
Putting~(\ref{estA}), (\ref{estB1}) and~(\ref{estB2}) together, we find
\begin{equation}\label{almostR_1}
\big| A \big| +\big| B \big| \leq C_5 = C_5(k,K,N)\,.
\end{equation}
We will prove the validity of the lemma with the following choice of $R_1$ and $\Gamma_1$,
\begin{align*}
R_1=2C_5+4\,, && \Gamma_1=2^{[K/\eta]+1}\,,
\end{align*}
where $C_5=C_5(k,K,N)$ and $\eta=\eta(k,K,N)$ have been introduced in~(\ref{almostR_1}) and in Lemma~\ref{threeconditions} respectively. To obtain our proof, we will distinguish the possible cases for $\Omega$.

\case{I}{One has $\hat t=-\infty$.}
If this case happens, then condition~(3) of Lemma~\ref{threeconditions} never holds true, i.e., for every $t\leq \bar t$ either condition~(1) or~(2) holds. Recalling the definition of $A$ and $B$ and~(\ref{almostR_1}), we deduce that $W(\Om,0,\mh)\leq C_5$. Therefore, the claim of Lemma~\ref{lemmatail} is immediately obtained simply taking $\Omh=\Omega$, since $R_1\geq C_5$ and $\Gamma_1\geq 1$.\par

\case{II}{One has $\hat t>-\infty$.}
In this case, let us notice that it must be $m(\hat t)> 0$, hence $(\hat t,\bar t)\subseteq A\cup B$ and thus by~(\ref{almostR_1}) $\hat t \geq \bar t - C_5$. Let us now pick some $t^\star \in [\hat t -1, \hat t]$ for which condition~(3) holds, and define $U_1:= \Omh(t^\star)$. By definition, $U_1$ has unit volume, and $\lambda_i(U_1)< \lambda_i(\Om)$ for every $1\leq i \leq k$, being condition~(3) true for $t^\star$.\par
Observe now that by definition for every $2\leq p \leq N$ one has $\pi_p\big(\Omt(t^\star)\big)= \pi_p\big(\Omega^+(t^\star)\big)$, hence
\[\begin{split}
{\rm diam} \big(\pi_p(U_1)\big) &= {\rm diam} \Big(\pi_p\big(\Omh(t^\star)\big)\Big)
={\rm diam} \Big(\pi_p \Big(\big| \Omt(t^\star) \big|^{-\frac 1N} \Omt(t^\star)\Big)\Big)
\leq 2 \, {\rm diam} \Big(\pi_p\big(\Omt(t^\star)\big)\Big)\\
&= 2 \, {\rm diam} \Big(\pi_p\big(\Omega^+(t^\star)\big)\Big)
\leq 2 \, {\rm diam} \big(\pi_p(\Omega)\big)\,,
\end{split}\]
where we have used that $\big| \Omt(t^\star)\big|\geq 1/2$. Concerning the widths of $U_1$ and $\Omega$, we can start observing that
\[
W\big(U_1,\mh,1\big) = \big| \Omt(t^\star) \big|^{-\frac 1N} \bigg( W\Big(\Omt(t^\star),\mh\big| \Omt(t^\star) \big|,\big| \Omt(t^\star) \big|\Big)\bigg)\,.
\]
Moreover, since it is admissible to assume $\nu^{\frac N{N-1}}< \frac \mh 2$ and then
\[
\big| \Omt(t^\star)^l_{t^\star} \big| = \big| \eps(t^\star)^{\frac N{N-1}}\big| < \big|\Omt(t^\star)\big| \mh\,,
\]
we have
\[
\tau\Big(\Omt(t^\star), \mh\big| \Omt(t^\star) \big|\Big) = \tau\Big(\Om, \mh\big| \Omt(t^\star)\big| + 1 - \big| \Omt(t^\star)\big|\Big)\,;
\]
as a consequence, we evaluate
\[\begin{split}
W\Big(\Omt(t^\star),\mh\big| \Omt(t^\star) \big|,\big| \Omt(t^\star) \big|\Big)
&=\tau\Big(\Omt(t^\star),\big|\Omt(t^\star)\big|\Big)-\tau\Big(\Omt(t^\star),\mh\big|\Omt(t^\star)\big|\Big)\\
&=\tau\big(\Om,1\big)-\tau\Big(\Om,\mh\big|\Omt(t^\star)\big|+1-\big|\Omt(t^\star)\big|\Big)
\leq\tau\big(\Om,1\big)-\tau\big(\Om,\mh\big)\\
&=W\big(\Om,\mh,1\big)\,,
\end{split}\]
thus we deduce, again recalling $\big| \Omt(t^\star)\big|\geq 1/2$, that
\[
W\big(U_1,\mh,1\big) \leq 2  W\big(\Om,\mh,1\big)\,.
\]

Summarizing, we have found that
\begin{align}\label{thistrue}
\lambda_i(U_1)< \lambda_i(\Omega)\,,
&&{\rm diam} \big(\pi_p(U_1)\big) \leq 2 \, {\rm diam} \big(\pi_p(\Omega)\big)\,,&&
 W\big(U_1,\mh,1\big) \leq 2 W\big(\Om,\mh,1\big)\,.
\end{align}
As a consequence, the choice $\Omh=U_1$ satisfies all the requirements of Lemma~\ref{lemmatail}, except possibly condition~(\ref{estimatetail}). To deal with this last condition, we need to further subdivide this case.

\case{IIa}{One has $\hat t>-\infty$ and $m(\hat t)< \mh$.}
In this case, we can show that the choice $\Omh=U_1$ actually works. As noticed above, we have only to prove the validity of~(\ref{estimatetail}). To do so, we assume for simplicity that $\bar t=0$, which is clearly admissible by translation. Hence, $t^\star \geq \hat t-1\geq \bar t - C_5 -1= -C_5-1$, and thus
\[
\Omt(t^\star) = \Omega^+(t^\star) \cup Q(t^\star) \subseteq \big\{ (x,y):\, x> t^\star - 1\big\}
\subseteq\big\{ (x,y):\, x> -C_5 -2\big\}\,.
\]
Recalling that $\big|\Omt(t^\star)\big| \geq \big| \Omega^+(t^\star)\big|\geq |\Omega^+(\bar t)|= 1 - 2\mh\geq 1/2$, we deduce that
\begin{equation}\label{finhe}
\Omh=\Omh(t^\star)= \big| \Omt(t^\star) \big|^{-\frac 1N} \Omt(t^\star) \subseteq 2 \,\Omt(t^\star)
\subseteq\big\{ (x,y):\, x> -2C_5 -4\big\}\,.
\end{equation}
Moreover, since $m(\hat t)< \mh$,
\[\begin{split}
\big|\Omh^l_0\big| &= \big|\Omh(t^\star)^l_0\big| \geq 
\big|\Omt(t^\star)^l_0\big| = \Big|\Big(\Omega^+(t^\star) \cup Q(t^\star)\Big)^l_0\Big|
\geq \Big|\Big(\Omega^+(t^\star) \Big)^l_0\Big|\\
&= \Big| \big\{ (x,y)\in\Omega:\, t^\star < x <0\big\}\Big|
=m(0) - m(t^\star) \geq m(0) - m(\hat t) \geq \mh\,,
\end{split}\]
and this implies that $\tau(\Omh,\mh)\leq 0$. The inclusion~(\ref{finhe}) ensures then that $\tau(\Omh,0)\geq -2 C_5 - 4$, and then~(\ref{estimatetail}) holds true since $R_1= 2C_5 +4$.

\case{IIb}{One has $\hat t>-\infty$ and $m(\hat t)\geq \mh$.}
We have now to face the last possible case, namely, when $\hat t$ is finite but $m(\hat t)\geq\mh$. In this case, thanks to Lemma~\ref{threeconditions} the estimates~(\ref{thistrue}) can be strengthened as
\begin{align}\label{thistrue2b}
\lambda_i(U_1)< \lambda_i(\Omega)-\eta\,,
&&{\rm diam} \big(\pi_p(U_1)\big) \leq 2 \, {\rm diam} \big(\pi_p(\Omega)\big)\,,&&
 W\big(U_1,\mh,1\big) \leq 2  W\big(\Om,\mh,1\big)\,.
\end{align}
Concerning the validity of~(\ref{estimatetail}), it does not follow by~(\ref{finhe}) because the assumption $m(\hat t)\geq \mh$ does not imply that $\tau(\Omh,\mh)\leq 0$. However, we can argue as follow: if~(\ref{estimatetail}) holds true for $U_1$, then of course we are done by setting $\Omh=U_1$. Otherwise, we apply the above construction to the set $U_1$ in place of $\Omega$: since $U_1$ does not satisfy~(\ref{estimatetail}), then Case~I is impossible, thus we are in Case~II and then by~(\ref{thistrue}) we find an open set $U_2$ of unit measure such that
\[\left\{\begin{array}{l}
\lambda_i(U_2)< \lambda_i(U_1) < \lambda_i(\Omega)-\eta \quad \forall\, 1\leq i \leq k\,, \\
{\rm diam} \big(\pi_p(U_2)\big) \leq 2 \, {\rm diam} \big(\pi_p(U_1)\big)\leq 4 \, {\rm diam} \big(\pi_p(\Omega)\big)\,,\\
W\big(U_2,\mh,1\big)\leq 2 W\big(U_1,\mh,1\big) \leq 4 W\big(\Om,\mh,1\big)\,.
\end{array}\right.\]
If $U_1$ is in Case~IIa, then as before we are done with the choice of $\Omh=U_2$. Otherwise, if we are in Case~IIb, then~(\ref{thistrue2b}) becomes
\begin{align*}
\lambda_i(U_2)< \lambda_i(\Omega)-2 \eta\,,
&&{\rm diam} \big(\pi_p(U_2)\big) \leq 4 \, {\rm diam} \big(\pi_p(\Omega)\big)\,,&&
 W\big(U_2,\mh,1\big) \leq 4  W\big(\Om,\mh,1\big)\,.
\end{align*}
Going on with the obvious iteration, if the proof has not been concluded after $\ell\in \N$ steps then we have found an open set $U_\ell$ satisfying
\begin{align*}
\lambda_i(U_\ell)< \lambda_i(\Omega)-\ell \eta\,,
&&{\rm diam} \big(\pi_p(U_\ell)\big) \leq 2^\ell \, {\rm diam} \big(\pi_p(\Omega)\big)\,,&&
 W\big(U_\ell,\mh,1\big) \leq 2^\ell  W\big(\Om,\mh,1\big)\,.
\end{align*}
This is of course impossible if $\ell \eta \geq K$, being $\lambda_k(\Omega)\leq K$: as a consequence, our iteration must stop after less than $K/\eta$ steps, thus our thesis is concluded with our choice of $\Gamma_1$.
\end{proof}

\subsection{Boundedness of the interior}

The goal of this subsection is to obtain a uniform bound also for the interior part of a set $\Omega$, in the sense of Lemma~\ref{lemmatail}. Most of the arguments of this case will be identical to those that we made for the tails in Section~\ref{sub1}, but some modifications are essential. 
In particular we give new definitions for $\eps,$ $\de_i$, $\mu_i$, $\Omt(t)$ and $\Omh(t)$ in order to mantain clear the analogy with what was done in Section~\ref{sub1}. 
The result that we are going to prove is the following.
\begin{lemma}\label{lemmainterior}
For every $K>0$ there exist $R_2=R_2(k,K,N)$ and $\Gamma_2=\Gamma_2(k,K,N)$ such that, for any open set $\Omega\subseteq \R^N$ of unit volume and with $\lambda_k(\Omega)\leq K$, and for any choice of $\om\in (\mh, 1- \frac \mh 2)$, there exists another open set $\Omh\subseteq\R^N$, still of unit volume, such that $\lambda_i(\Omh)\leq \lambda_i(\Om)$ for every $1\leq i \leq k$, and such that for every $2\leq p \leq N$
\begin{align}\label{estimateinterior}
W\big(\Omh,0,\om\big) \leq R_2+\Gamma_2 W\big(\Om, 0, \om- \mh\big)\,, &&
{\rm diam} \big(\pi_p(\Omh)\big)\leq \Gamma_2\, {\rm diam} \big(\pi_p(\Omega)\big)\,.
\end{align}
\end{lemma}
To start with, we give the analogous of the definitions~(\ref{int0}), (\ref{int1}) and~(\ref{int2}) of Section~\ref{sub1} that we need now; Figure~\ref{intern} helps to visualize the new situation. More precisely, we set for brevity
\begin{align*}
t_0 := \frac{\tau(\Om,\om+\frac \mh 2)+\tau(\Om,\om-\mh)}2\,, && \bar t:=\frac{\tau(\Om,\om+\frac \mh 2)-\tau(\Om,\om-\mh)}2\,;
\end{align*}
keep in mind that, since $\om\in (\mh,1-\frac\mh 2)$, then $-\infty < \tau(\Om,\om-\mh)<\tau(\Om,\om+\frac \mh 2) < +\infty$. For any $0\leq t \leq \bar t$, we define
\begin{align*}
\Om^+(t)&:= \Om^l_{t_0-t} \cup \Om^r_{t_0+t}\,, & \Omega^-(t) &:= \Om^r_{t_0-t} \cap \Om^l_{t_0+t}=\Omega\setminus\Om^+(t)\,, \\
\eps(t)&:=\hc^{N-1}(\Om_{t_0-t})+\hc^{N-1}(\Om_{t_0+t})\,, &
m(t) &:= \big| \Omega^-(t) \big| = \int_0^t \eps(s)\,ds\leq \frac 32\,\mh\,.
\end{align*}

Moreover, having fixed an orthonormal set $\big\{ u_1,\, u_2,\, \dots\,,\, u_k\big\}$ of eigenfunctions with unit $L^2$ norm corresponding to the first $k$ eigenvalues of $\Omega$, for every $1\leq i \leq k$ and $0\leq t \leq \bar t$ we define
\begin{align*}
\ \de_i(t):=\int_{\Om_{t_0-t}} |Du_i|^2+\int_{\Om_{t_0+t}} |Du_i|^2\,, && 
\mu_i(t):=\int_{\Om_{t_0-t}} u_i^2+\int_{\Om_{t_0+t}} u_i^2\,.\ 
\end{align*}
In analogy with~(\ref{defphi}), we define again $\delta(t)=\sum_{i=1}^k \delta_i(t)$, and we set again
\[
\phi(t) := \sum_{i=1}^k \int_{\Omega^-(t)}  |Du_i|^2 = \int_0^t \delta(s)\,ds \,.
\]
\begin{figure}[htbp]
\begin{center}
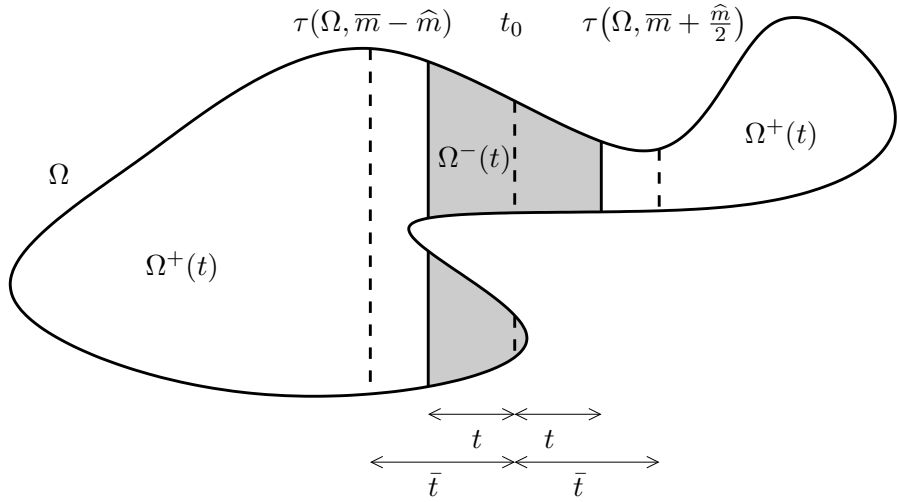
\caption{A set $\Omega$ and the corresponding quantities $t_0,\, \bar t$ and sets $\Om^+(t)$ (white) and $\Om^-(t)$ (shaded).}
\label{intern}
\end{center}
\end{figure}
Our strategy to prove Lemma~\ref{lemmainterior} is very similar to what we did to show Lemma~\ref{lemmatail}; in fact, basically the only difference is that to show the analogous of Lemma~\ref{primastima} we cannot rely on the symmetrization of $\Omega^-(t)$. Let us see how to overcome this difficulty.

\begin{lemma}\label{newass}
There exists a small constant $\nu=\nu(k,K,N)<1$ such that, if $\Omega$ and $\om$ are as in the assumptions of Lemma~\ref{lemmainterior}, and $0\leq t \leq \bar t$ is such that $\eps(t),\, \delta(t)\leq \nu$, then for every $1\leq i \leq k$ one has 
\begin{align}\label{2udu-}
\int_{\Om^-(t)} u_i^2 \leq C \eps(t)^{\frac{1}{N-1}}\de_i(t)\,, &&
\int_{\Om^-(t)} |Du_i|^2 \leq C\eps(t)^{\frac{1}{N-1}}\de_i(t)\,.
\end{align}
\end{lemma}
\begin{proof}
Consider the ``external cylinders''
\begin{align*}
Q_1 := \big(t_0-t-\sigma_1, t_0-t\big) \times \Omega_{t_0-t}\,, &&
Q_2 := \big(t_0+t,t_0+t+\sigma_2\big) \times \Omega_{t_0+t}\,,
\end{align*}
where
\begin{align*}
\sigma_1= \hc^{N-1}(\Om_{t_0-t})^{\frac 1{N-1}}\,, &&
\sigma_2= \hc^{N-1}(\Om_{t_0+t})^{\frac 1{N-1}}\,,
\end{align*}
in perfect analogy with~(\ref{defcyl}) and~(\ref{defsigma}). Calling $U=\Omega^-(t)\cup Q_1 \cup Q_2$, we can extend~(\ref{utilde}) to obtain the following definition of $\ut_i \in W^{1,2}_0(U)$,
\[
\ut_i(x,y):=\left\{
\begin{array}{ll}
u_i(x,y) &\hbox{if $(x,y)\in \Omega^-(t)$}\,, \\[5pt]
\bal\frac{x-(t_0-t-\sigma_1)}{\sigma_1}\,u_i(t_0-t,y)\eal\qquad &\hbox{if $(x,y)\in Q_1$}\,, \\[10pt]
\bal\frac{(t_0+t+\sigma_2)-x}{\sigma_2}\,u_i(t_0+t,y)\eal &\hbox{if $(x,y)\in Q_2$}\,.
\end{array}
\right.
\]
Applying Lemma~\ref{lemmatest} to the two cylinders $Q_1$ and $Q_2$, and comparing the present definitions of $\eps$ and $\delta_i$ with those that we used in Section~\ref{sub1}, (\ref{newtest}) gives us
\[
\int_{Q_1\cup Q_2}|D\ut_i|^2 \leq C_2\eps(t)^{\frac{1}{N-1}}\de_i(t)\,.
\]
We can then obtain an estimate between $\int_{\Om^-(t)} u_i^2$ and $\int_{\Om^-(t)} |Du_i|^2$ similar to~(\ref{eq1}), first noticing that
\[
\frac{\lambda_1(B_N)}{\big(m(t)+\eps(t)^{\frac N{N-1}}\big)^{\frac{2}{N}}} \leq \frac{\lambda_1(B_N)}{|U|^{\frac{2}{N}}}
\leq \la_1(U) \leq \rc(\ut_i,U)
\leq \frac{\bal \int_{\Om^-(t)} |Du_i|^2 + \int_{Q_1\cup Q_2} |D\ut_i|^2\eal}{\bal \int_{\Om^-(t)} u_i^2 \eal}\,,
\]
and then deducing, recalling~(\ref{defmh}) and choosing $\nu$ small enough,
\begin{equation}\label{wtebte}\begin{split}
\int_{\Om^-(t)} u_i^2 &\leq \frac{\big(m(t)+\eps(t)^{\frac N{N-1}}\big)^{\frac 2N}}{\lambda_1(B_N)}\, \bigg(\int_{\Om^-(t)} |Du_i|^2 + C_2\eps(t)^{\frac{1}{N-1}}\de_i(t) \bigg)\\
&\leq \frac{1}{2^{1+\frac 1N}K} \int_{\Om^-(t)} |Du_i|^2 +\frac{C_2}{2^{1+\frac 1N}K}\,\eps(t)^{\frac{1}{N-1}}\de_i(t)\,.
\end{split}\end{equation}
Observe that $\nu=\nu(k,K,N)$ can be chosen so small that the last estimate implies 
\[
\int_{\Om^-(t)} u_i^2 \leq \frac 12\,.
\]
We can also generalize~(\ref{eq2}); in fact, since the very same argument used in~(\ref{muest}) again ensures that $\mu_i(t)\leq C \eps(t)^{\frac 2{N-1}} \delta_i(t)$, we can obtain
\begin{equation}\label{wtebte2}\begin{split}
\int_{\Om^-(t)} |Du_i|^2 &=  \int_{\Om^-(t)} \lambda_i u_i^2 + \int_{\Om_{t_0-t}\cup \Om_{t_0+t}}  u_i \,\frac{\partial u_i}{\partial \nu}
\leq K \int_{\Om^-(t)} u_i^2 + \sqrt{\mu_i(t)\delta_i(t)}\\
&\leq K \int_{\Om^-(t)} u_i^2 + C \eps(t)^{\frac 1{N-1}} \delta_i(t)\,.
\end{split}\end{equation}
Putting together~(\ref{wtebte}) and~(\ref{wtebte2}) gives~(\ref{2udu-}).
\end{proof}
We need now to extend the result of Lemma~\ref{lemmanu} to our new setting. To do so, going on in analogy with Section~\ref{sub1}, we give the following definition.
\begin{defin}
Let $\Omega$, $\om$ and $1\leq t \leq \bar t$ be as in the assumptions of Lemma~\ref{newass}. Consider the ``internal cylinders''
\begin{align*}
Q_1 := \big(t_0-t, t_0-t+\sigma_1\big) \times \Omega_{t_0-t}\,, &&
Q_2 := \big(t_0+t-\sigma_2,t_0+t\big) \times \Omega_{t_0+t}\,,
\end{align*}
where
\begin{align*}
\sigma_1= \hc^{N-1}(\Om_{t_0-t})^{\frac 1{N-1}}\,, &&
\sigma_2= \hc^{N-1}(\Om_{t_0+t})^{\frac 1{N-1}}\,,
\end{align*}
and notice that by the assumption on $\eps(t)$ and the fact that $t\geq 1$ one has $Q_1\cap Q_2=\emptyset$. The set $\Omt(t)$ is defined as
\[\begin{split}
\Omt(t):= \Big\{(x,y)\in \R^N: \hbox{either } & x\leq t_0,\, \big(x-t+\sigma_1,y\big)\in \Om^+(t)\cup Q_1\,, \\
\hbox{or } & x\geq t_0,\, \big(x+t-\sigma_2,y\big)\in \Om^+(t)\cup Q_2\Big\}\,,
\end{split}\]
see Figure~\ref{newOmtt}. 
Notice that
\[\begin{split}
\big| \Omt(t)\big| &= \big| \Om^+(t)\big| + \big| Q_1 \big| + \big| Q_2 \big| = 1 - m(t) + \H^{N-1}\big(\Omega_{t_0-t}\big)^{\frac N{N-1}}+ \H^{N-1}\big(\Omega_{t_0+t}\big)^{\frac N{N-1}}\\
&\leq 1 - m(t) +\eps(t)^{\frac N{N-1}}\,.
\end{split}\]
Moreover, define again the rescaled set
\[
\Omh(t) := \big| \Omt(t) \big|^{-\frac 1N} \Omt(t)\,.
\]
\end{defin}

\begin{figure}[htbp]
\begin{center}
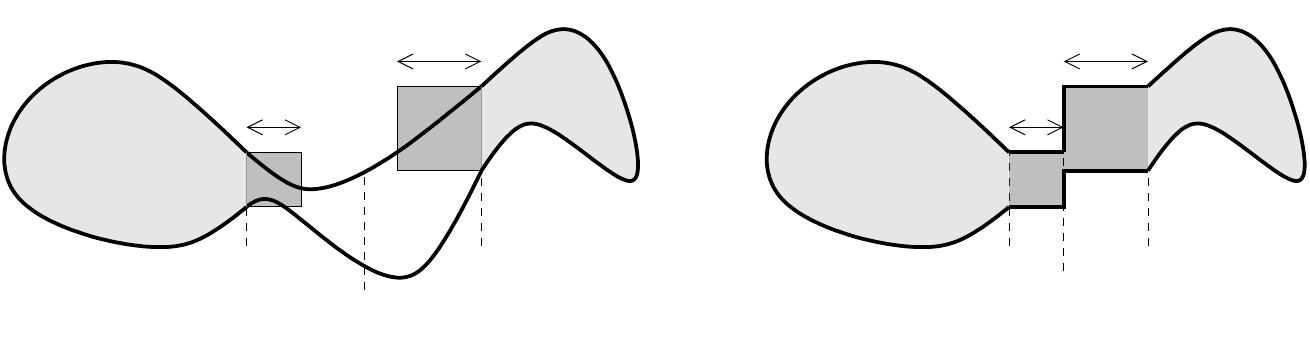
\caption{A set $\Omega$ and the corresponding set $\Omt(t)$.}
\label{newOmtt}
\end{center}
\end{figure}

With this definition of the sets $\Omt(t)$ and with the obvious extension of (\ref{utilde}) in order to define $\ut_i\in W^{1,2}_0(\Omt(t))$, we can now literally repeat the proofs of Lemmas~\ref{lemmatest}, \ref{noth}, \ref{lemmanu} and~\ref{threeconditions}, the unique difference being the substitution of $Q(t)$ with $Q_1\cup Q_2$, and of $\Omega_t$ with $\Omega_{t_0+t}\cup \Omega_{t_0-t}$. We obtain then the following result, which holds up to possibly decrease the constant $\nu=\nu(k,K,N)$ of Lemma~\ref{newass}.

\begin{lemma}\label{threebis}
Let $\Omega$ be as in the assumptions of Lemma~\ref{lemmainterior}, and let $1\leq t\leq \bar t$. There exists $C_6=C_6(k,K,N)$ such that exactly one of the three following conditions hold:
\begin{enumerate}
\item[(1) ] $\max\big\{ \eps(t),\, \delta(t) \big\} > \nu$;
\item[(2) ] (1) does not hold and $m(t) \leq C_6 \big( \eps(t) + \delta(t)\big) \eps(t)^{\frac 1{N-1}}$;
\item[(3) ] (1) and~(2) do not hold and for every $1\leq i \leq k$, one has $\lambda_i\big(\Omh(t)\big)< \lambda_i(\Omega)$.
\end{enumerate}
In particular, if condition~(3) holds for $t$ and $m(t)\geq \mh/2$, then for every $1\leq i \leq k$ one has $\lambda_i\big(\Omh(t)\big) < \lambda_i(\Om) - \eta$, being $\eta=\eta(k,K,N)>0$.
\end{lemma}

We can now conclude this section by presenting the proof of Lemma~\ref{lemmainterior}, which will be a minor modification of the proof of Lemma~\ref{lemmatail}.

\begin{proof}[Proof of Lemma~\ref{lemmainterior}]
First of all, we want to show that it is admissible to assume
\begin{align}\label{admissible}
m(t)>0 && \forall \, t>0\,.
\end{align}
In fact, suppose that it is not so, and let $\tau=\max \{ 0\leq t\leq \bar t: m(t)=0 \}>0$. Then, $\Omega$ is the disjoint union of $\Omega \cap \{x> t_0+\tau\}$ and $\Omega \cap \{x< t_0-\tau\}$, and it does not intersect the whole strip $\{t_0-\tau<x<t_0+\tau\}$. Therefore, replacing $\Omega$ with $\big\{(x+\tau,y):\, x<t_0, (x,y)\in\Om\big\}\cup \big\{(x-\tau,y):\, x>t_0, (x,y)\in\Om\big\}$, that is, moving closer the two disjoint parts of $\Omega$, does not change any of the eigenvalues of $\Omega$ and is clearly admissible for the proof of the lemma; moreover, the property~(\ref{admissible}) of course holds true for this new set. Hence, from now on we directly assume that~(\ref{admissible}) holds true for $\Omega$.\par
Define now $\hat t$ analogously to~(\ref{defhatt}) by setting
\[
\hat t := \sup \Big\{ 1 \leq t \leq \bar t : \, \hbox{condition~(3) of Lemma~\ref{threebis} holds for $t$}\Big\}\,,
\]
with the convention that, if condition (3) is false for every $1\leq t\leq \overline{t},$ then $\widehat{t}=1$. Again we define $A$ and $B$ as
\begin{align*}
A:&= \Big\{ t\in (\hat t,\bar t):\, \hbox{condition~(1) of Lemma~\ref{threebis} holds for $t$} \Big\}\,,\\
B:&=\Big\{t\in (\hat t,\bar t):\, \hbox{condition~(2) of Lemma~\ref{threebis} holds for $t$ and $m(t)>0$} \Big\}\,.
\end{align*}
The same argument of the proof of Lemma~\ref{lemmatail} gives then
\begin{equation}\label{almostR_2}
\big| A \big| +\big| B \big| \leq C_7 = C_7(k,K,N)\,,
\end{equation}
and we are going to show the thesis with the choice
\begin{align*}
R_2=4C_7+8\,, && \Gamma_2=2^{[K/\eta]+1}\,,
\end{align*}
being $\eta$ the constant of Lemma~\ref{threebis}. We can again subdivide the possible cases for $\Omega$.
\case{I}{One has $\hat t=1$.}
In this case, by~(\ref{admissible}) one has that every $1<t\leq \bar t$ belongs either to $A$ or to $B$, thus by~(\ref{almostR_2}) $\bar t \leq C_7+1$; as a consequence, the choice $\Omh=\Om$ satisfies the requirements of the lemma. Indeed, while the right condition of~(\ref{estimateinterior}) is obviously true, the left one follows just noticing that
\[\begin{split}
W\big(\Omh,0,\om\big) &=
W\big(\Om,0,\om\big)
=\tau(\Om,\om)-\tau(\Om,\om-\mh)+W\Big(\Om,0,\om-\mh\Big)\\
&\leq \tau\Big(\Om,\om+\frac \mh 2\Big)-\tau(\Om,\om-\mh) +W\Big(\Om,0,\om-\mh\Big)
= 2\bar t+ W\Big(\Om,0,\om-\mh\Big)\\
&\leq 2 \big( C_7+1\big)+W\Big(\Om,0,\om-\mh\Big)\,.
\end{split}\]
\case{II}{One has $\hat t>1$.}
In this case, again by~(\ref{admissible}) we know that $A\cup B$ contains the whole segment $(\hat t,\bar t)$, thus $\bar t \leq \hat t + C_7$ by~(\ref{almostR_2}). If we choose $t^\star \in (\hat t-1,\hat t)$ for which condition~(3) holds and define $U_1:= \Omh(t^\star)$, we know by construction that $|U_1|=1$ and $\lambda_i(U_1)< \lambda_i(\Om)$ for every $1\leq i \leq k$. As in the proof of Lemma~\ref{lemmatail}, the fact that $|\Omt(t^\star)|\geq 1-\frac 32\mh \geq 1/2$ ensures that ${\rm diam} \big(\pi_p(U_1)\big)\leq 2 \, {\rm diam} \big(\pi_p(\Omega)\big)$ for each $2\leq p \leq N$. On the other hand, concerning the width of $U_1$, recalling the definition of $\Omt(t^\star)$ and observing that $(\om-\mh)\big| \Omt(t^\star) \big| < \om -\mh$ one finds
\[\begin{split}
W\Big(U_1, 0, \om-\mh\Big) &=
W\Big(\Omh(t^\star), 0, \om-\mh\Big)
=\big| \Omt(t^\star) \big|^{-\frac 1N}W\Big(\Omt(t^\star), 0, (\om-\mh)\big| \Omt(t^\star) \big|\Big)\\
&\leq 2\bigg(\tau\Big(\Omt(t^\star), (\om-\mh)\big| \Omt(t^\star) \big|\Big)-\tau\Big(\Omt(t^\star), 0\Big)\bigg)\\
&=2\bigg(\tau\Big(\Om, (\om-\mh)\big| \Omt(t^\star) \big|\Big)-\tau\Big(\Om, 0\Big)\bigg)
=2 W\Big(\Om, 0, (\om-\mh)\big| \Omt(t^\star) \big|\Big)\\
&\leq 2 W\Big(\Om, 0, \om-\mh\Big)\,.
\end{split}\]
Summarizing, we have found that
\begin{align}\label{thistrueint}
\lambda_i(U_1)< \lambda_i(\Omega)\,,
&&\frac{{\rm diam} \big(\pi_p(U_1)\big)}{{\rm diam} \big(\pi_p(\Omega)\big)} \leq 2 \,,&&
\frac{W\big(U_1,0,\om-\mh\big)}{W\big(\Om,0,\om-\mh\big)} \leq 2\,.
\end{align}
To conclude the inspection of the validity of~(\ref{estimateinterior}), we will again need to consider separately two subcases.

\case{IIa}{One has $\hat t>1$ and $m(t^\star)\leq \mh/2$.}
In this case, we can quickly observe that
\begin{equation}\label{ifcaseIIa}\begin{split}
W\big(U_1,0,\om\big)&=
\big| \Omt(t^\star) \big|^{-\frac 1N}W\Big(\Omt(t^\star), 0, \om\big| \Omt(t^\star) \big|\Big)
\leq 2W\Big( \Omt(t^\star), 0, \om + \frac \mh 2 - \big| \Om^-(t^\star)\big|\Big) \\
&\leq 2W\big(\Om, 0 , \om-\mh\big) + 4\big(\bar t - t^\star+1\big)
\leq  2W\big(\Om, 0 , \om-\mh\big) + 4\big(C_7 + 2\big)\,,
\end{split}\end{equation}
and, together with~(\ref{thistrueint}), this concludes the proof of~(\ref{estimateinterior}) and of the lemma with the choice $\Omh=U_1$.

\case{IIb}{One has $\hat t>1$ and $m(t^\star)>\mh/2$.}
Let us conclude with this last case. By Lemma~\ref{threebis}, in place of~(\ref{thistrueint}) we have then
\begin{align*}
\lambda_i(U_1)< \lambda_i(\Omega)-\eta\,,
&&\frac{{\rm diam} \big(\pi_p(U_1)\big)}{{\rm diam} \big(\pi_p(\Omega)\big)} \leq 2 \,,&&
\frac{W\big(U_1,0,\om-\mh\big)}{W\big(\Om,0,\om-\mh\big)} \leq 2\,,
\end{align*}
which still does not guarantee the validity of~(\ref{estimateinterior}). However, we can apply the construction above to $U_1$: if $U_1$ is in Case~I, then the choice $\Omh=U_1$ concludes the proof; otherwise, there exists an open set $U_2$ of unit measure satisfying
\begin{align}\label{thisU_2}
\lambda_i(U_2)< \lambda_i(\Omega)-\eta\,,
&&\frac{{\rm diam} \big(\pi_p(U_2)\big)}{{\rm diam} \big(\pi_p(\Omega)\big)} \leq 4 \,,&&
\frac{W\big(U_2,0,\om-\mh\big)}{W\big(\Om,0,\om-\mh\big)} \leq 4\,.
\end{align}
If $U_1$ is in Case~IIa then~(\ref{ifcaseIIa}) gives
\[
W\big(U_2,0,\om\big)\leq 2W\big(U_1, 0 , \om-\mh\big) + 4\big(C_7 + 2\big)
\leq 4 W\big(\Om, 0 , \om-\mh\big) + 4\big(C_7 + 2\big)\,,
\]
so the choice $\Omh=U_2$ concludes the proof. Instead, if $U_1$ is in Case~IIb then the first inequality of~(\ref{thisU_2}) becomes $\la_i(U_2)<\la_i(\Om)-2\eta$ for every $1\leq i \leq k$. The obvious iteration ensures us that, if the proof has not been obtained after $\ell$ steps, then there must be some open set $U_\ell$ satisfying
\begin{align*}
\lambda_i(U_\ell)< \lambda_i(\Omega)-\ell \eta\,, &&
\frac{{\rm diam} \big(\pi_p(U_\ell)\big)}{{\rm diam} \big(\pi_p(\Omega)\big)} \leq 2^\ell \,,&&
\frac{W\big(U_\ell,0,\om-\mh\big)}{W\big(\Om,0,\om-\mh\big)} \leq 2^\ell\,.
\end{align*}
Since this is not possible for $\ell>K/\eta$, the iteration must stop at some $\ell\leq [K/\eta]$, and thus we conclude the proof thanks to the choice of $\Gamma_2$.
\end{proof}

\subsection{Proof of Proposition~\ref{allhere}}

We are finally in position to give the proof of Proposition~\ref{allhere}, which is now a simple consequence of Lemma~\ref{lemmatail} and Lemma~\ref{lemmainterior}.

\begin{proof}[Proof of Proposition~\ref{allhere}]
Let us pick a generic open set $\Omega$ with $\lambda_k(\Omega)\leq K$. Applying Lemma~\ref{lemmatail} to $\Omega$, we find a set $E_1$ with
\begin{align*}
\la_i(E_1)\leq \la_i(\Om)\,, &&
W(E_1,0,\mh) \leq R_1\,, &&
\frac{{\rm diam} \big(\pi_p(E_1)\big)}{{\rm diam} \big(\pi_p(\Om)\big)} \leq \Gamma_1 \,,
\end{align*}
for every $2 \leq p \leq N$. Then, we apply Lemma~\ref{lemmainterior} to $E_1$ with $\om=2\mh$ finding $E_2$ which satisfies
\begin{align*}
\la_i(E_2)\leq \la_i(\Om)\,, &&
W(E_2,0,2\mh) \leq R_2 + \Gamma_2 R_1\,, &&
\frac{{\rm diam} \big(\pi_p(E_2)\big)}{{\rm diam} \big(\pi_p(\Om)\big)} \leq \Gamma_2\Gamma_1 \,.
\end{align*}
Iterating, for any $\ell\geq 3$ such that $\ell\mh\leq 1- \frac \mh 2$ we apply Lemma~\ref{lemmainterior} to $E_{\ell-1}$ with $\om=\ell\mh$ finding $E_\ell$ such that
\begin{align*}
\la_i(E_\ell)\leq \la_i(\Om)\,, &&
W(E_\ell,0,\ell\mh) \leq R_2 \frac{\Gamma_2^{\ell-1}-1}{\Gamma_2-1}+ \Gamma_2^{\ell-1} R_1\,, &&
\frac{{\rm diam} \big(\pi_p(E_\ell)\big)}{{\rm diam} \big(\pi_p(\Om)\big)} \leq \Gamma_2^{\ell-1}\Gamma_1 \,.
\end{align*}
Possibly applying a last time Lemma~\ref{lemmainterior} with $\om=1-\mh$, we have then found an open set $E$ satisfying
\begin{align*}
\la_i(E)\leq \la_i(\Om)\,, &&
W(E,0,1-\mh) \leq R_2 \frac{\Gamma_2^{[1/\mh]-1}-1}{\Gamma_2-1}+ \Gamma_2^{[1/\mh]-1} R_1 , &&
\frac{{\rm diam} \big(\pi_p(E)\big)}{{\rm diam} \big(\pi_p(\Om)\big)} \leq \Gamma_2^{[1/\mh]-1}\Gamma_1 .
\end{align*}
Calling $E'$ the set obtained by reflecting $E$ with respect to the plane $\{x=0\}$, the above estimates become
\begin{align*}
\la_i(E')\leq \la_i(\Om)\,, &&
W(E',\mh,1) \leq R_2 \frac{\Gamma_2^{[1/\mh]-1}-1}{\Gamma_2-1}+ \Gamma_2^{[1/\mh]-1} R_1\,, &&
\frac{{\rm diam} \big(\pi_p(E')\big)}{{\rm diam} \big(\pi_p(\Om)\big)} \leq \Gamma_2^{[1/\mh]-1}\Gamma_1 \,,
\end{align*}
so that applying once again Lemma~\ref{lemmatail} to $E'$ we find a set $F_1$ satisfying
\begin{align*}
\la_i(F_1)\leq \la_i(\Om)\,, && {\rm diam} \big(\pi_1(F_1)\big)=W(F_1,0,1) \leq R_3\,, &&
\frac{{\rm diam} \big(\pi_p(F_1)\big)}{{\rm diam} \big(\pi_p(\Om)\big)} \leq \Gamma_3 \,,
\end{align*}
having set
\begin{align*}
R_3 := R_1+\Gamma_1R_2 \frac{\Gamma_2^{[1/\mh]-1}-1}{\Gamma_2-1}+ \Gamma_1\Gamma_2^{[1/\mh]-1} R_1\,, && \Gamma_3 := \Gamma_2^{[1/\mh]-1}\Gamma_1^2\,.
\end{align*}
We can now repeat the whole construction, using as starting set $F_1$ in place of $\Omega$, and using the second coordinate in place of the first one. We will end up with a set $F_2$ with
\begin{align*}
\la_i(F_2)\leq \la_i(F_1)\leq \la_i(\Om)\quad  \forall \,1\leq i \leq k\,, &&
{\rm diam} \big(\pi_2(F_2)\big)\leq R_3\,,
\end{align*}
and such that for every $p\neq 2$ it is ${\rm diam} \big(\pi_p(F_2)\big)\leq \Gamma_3\,{\rm diam} \big(\pi_p(F_1)\big) $. In particular, choosing $p=1$ we discover that ${\rm diam} \big(\pi_1(F_2)\big)\leq \Gamma_3 R_3$. We have now to iterate also this argument: for any $3\leq j \leq N$ we repeat the above construction starting from $F_{j-1}$ and using the $j$-th coordinate in the whole procedure, obtaining a set $F_j$ which satisfies
\begin{align*}
\la_i\big(F_j\big)\leq \la_i(\Om)\quad \forall 1\leq i \leq k\,, && {\rm diam} \big(\pi_p(F_j)\big)\leq \Gamma_3^{j-p} R_3\quad \forall\, 1 \leq p \leq j \,.
\end{align*}
The thesis is then finally obtained by defining $\Omh:=F_N$, being $R:= \Gamma_3^{N-1} R_3$.
\end{proof}

\section{Proof of Theorem~\mref{bounded}\label{sec3}}

This last section is devoted to the proof of Theorem~\mref{bounded}. For the ease of presentation, we will begin with a couple of technical lemmas, then passing to the proof of the theorem. It is important to observe that the core of our construction, namely Lemma~\ref{startingcase}, is in fact an easy consequence of the well-known result by Ashbaugh--Benguria in~\cite{AB}, which states that the ratio $\lambda_2/\lambda_1$ is bounded. However, we prefer to give a formal proof of this lemma to keep the present paper self-contained.\par
The first simple step of our construction states that functions with bounded Rayleigh quotients cannot concentrate too much on small regions.

\begin{lemma}\label{smallrho}
For every $m\in (0,1]$ and $K>0$ there exists $\rho=\rho(m,K,N)>0$ such that the following holds. Let $u\in W^{1,2}(\R^N)$ with
\begin{align*}
\int_{\R^N} u^2 = 1\,, && \int_{\R^N} |Du|^2 \leq K\,.
\end{align*}
Then for every cube $Q\subseteq \R^N$ with half-side $\rho$ one has
\[
\int_Q u^2 \leq m\,.
\]
\end{lemma}
\begin{proof}
Suppose that the claim is not true. Then there exists a sequence $\{u_n\}\subseteq W^{1,2}(\R^N)$ satisfying
\begin{align}\label{asspt}
\int_{\R^N} u_n^2 = 1\,, && \int_{\R^N} |Du_n|^2 \leq K\,, && \int_{Q_{1/n}} u_n^2 \geq m\,,
\end{align}
being $Q_r=[-r,r]^N$ the cube of half-side $r$ centered at the origin. By definition, this sequence is bounded in $W^{1,2}(\R^N)$, hence up to a subsequence we have that $u_n$ weakly converges to some function $u\in W^{1,2}(\R^N)$. In particular, for any $\eps>0$, $u_n$ strongly converges to $u$ in $L^2(Q_\eps)$, so that, thanks to~(\ref{asspt}), one has $\int_{Q_\eps} u^2 \geq m$. Since this is absurd, the claim follows.
\end{proof}

The second lemma, which is the core of our proof of Theorem~\mref{bounded}, ensures that every set with bounded first eigenvalue can be split into two subregions, each of them having first eigenvalue not too large.

\begin{lemma}\label{startingcase}
For every $K>0$ there exists $K'=K'(K,N)$ such that, if $\Omega$ is an open subset of $\R^N$ with $\lambda_1(\Omega)\leq K$, then there are two disjoint open subsets $\Omega_1,\, \Omega_2$ of $\Omega$ with $\lambda_1(\Omega_i) \leq K'$ for $i=1,\,2$.
\end{lemma}
\begin{proof}
We start applying Lemma~\ref{smallrho} with $K$ and with $m=1/2$, thus getting a positive number $\rho$. Let then $\Omega\subseteq \R^N$ be an open set with $\lambda_1(\Omega)\leq K$, and let $u\in W^{1,2}_0(\Omega)$ be a first eigenfunction of $\Omega$ with unit $L^2$ norm. Extending $u$ by $0$ outside $\Omega$, we have then by definition
\begin{align}\label{firstprop}
\int_{\R^N} u^2 = \int_\Omega u^2 = 1 \,, && \int_{\R^N} |Du|^2 = \int_\Omega |Du|^2 \leq K \,.
\end{align}
Let now $t^-<t^+$ be identified by
\begin{equation}\label{deft-}
\int_{\Omega^l_{t^-}} u^2 = \int_{\Omega^r_{t^+}} u^2 = \frac{1}{4N}\,.
\end{equation}
We claim that it is possible to assume
\begin{equation}\label{admass}
t^+ - t^- \geq 2 \rho\,.
\end{equation}
In fact, if it is not so, this means that there is a vertical stripe of width $2\rho$ out of which the squared $L^2$ norm of $u$ is less than $1/(2N)$ (by ``vertical'' we mean orthogonal to $e_1$). If this happens for every direction $e_1,\, e_2,\, \dots \,,\, e_N$, the intersection of the corresponding stripes is a square of half-side $\rho$ out of which the squared $L^2$ norm of $u$ is less than $1/2$. Since this is in contradiction with Lemma~\ref{smallrho}, we obtain the validity of~(\ref{admass}), up to a rotation.\par
Let us now call $t=(t^+ + t^-)/2$, define $\Omega_1 = \Omega^l_t$ and $\Omega_2 = \Omega^r_ t$, and let $\ut\in W^{1,2}(\Omega_1)$ be defined as
\[
\ut(x,y) := \left\{\begin{array}{ll}
u(x,y)  &\hbox{for $x\leq t -\rho$}\,, \\[5pt]
\bal\frac{t -x}{\rho}\eal\, u(x,y) \qquad &\hbox{for $t - \rho\leq x \leq t$}\,. \\
\end{array}\right.
\]
Since $u\in W^{1,2}_0(\Omega)$, it is clear that $\ut\in W^{1,2}_0(\Omega_1)$. Moreover, writing $Du=(D_1 u,\, D_y u)$, one has
\[
D\ut(x,y) = \bigg( \frac{t-x}{\rho} \,D_1 u(x,y) -\frac 1\rho\, u(x,y), \frac{t-x}{\rho}\, D_y u (x,y)\bigg)
\]
for every $(x,y)\in\Omega_1$ with $x\geq t-\rho$. As a consequence, minding~(\ref{firstprop}) one gets
\begin{equation}\label{estimDu}
\int_{\Omega_1} |D\ut|^2 \leq 2\int_{\Omega_1} |Du|^2 + \frac 2{\rho^2} \int_{\Omega_1} u^2 \leq 2K + \frac 2{\rho^2}\,.
\end{equation}
On the other hand, recalling~(\ref{admass}) and~(\ref{deft-}) it is
\begin{equation}\label{estimu}
\int_{\Omega_1} \ut^2 \geq \int_{\Omega^l_{t^-}} u^2 = \frac{1}{4N}\,.
\end{equation}
Putting together~(\ref{estimDu}) and~(\ref{estimu}) one immediately obtains
\[
\lambda_1(\Omega_1) \leq \rc(\ut,\Omega_1) = \frac{\bal\int_{\Omega_1} |D\ut|^2\eal}{\bal\int_{\Omega_1} \ut^2\eal}
\leq 8N \bigg( K + \frac 1{\rho^2}\bigg)\,.
\]
Finally, we can set $K' = 8N\big( K + 1/\rho^2\big)$: since we have shown that $\lambda_1(\Omega_1)\leq K'$, and since by symmetry it is also $\lambda_1(\Omega_2)\leq K'$, the thesis follows.
\end{proof}

We are now in position to prove a first boundedness result of $\lambda_k$ in terms of $\lambda_1$, from which Theorem~\mref{bounded} will then readily follow.

\begin{lemma}\label{exthm}
For every $K>0$ there exists $M'=M'(k,K,N)>0$ such that, for all open sets $\Om\subseteq\R^N$, if $\la_1(\Om)\leq K$ then $\la_k(\Om)\leq M'$.
\end{lemma}
\begin{proof}
Let us start by setting $K_1=K$, and then, applying Lemma~\ref{startingcase}, we let recursively $K_{l+1}=K'(K_l,N)$ for every $l\geq 1$. Finally, we define $M'=K_{j+1}$, where $j$ is the smallest natural number such that $2^j \geq k$. We will show the claim of the theorem with such constant $M'$.\par
To do so, we pick any open set $\Omega$ with $\lambda_1(\Omega)\leq K=K_1$. Applying Lemma~\ref{startingcase} to $\Omega$ with constant $K_1$, we find two disjoint open sets $\Omega_1,\,\Omega_2\subseteq\Omega$ with $\lambda_1(\Omega_i)\leq K'(K_1,N)=K_2$ for $i=1,\, 2$. Then, we can apply Lemma~\ref{startingcase} to $\Omega_1$ and $\Omega_2$ with constant $K_2$, finding four disjoint subsets $\Omega_{11},\, \Omega_{12},\, \Omega_{21},\, \Omega_{22}$ of $\Omega$, each of them with first eigenvalue smaller than $K_3$. Continuing with the obvious induction, we end up with $2^j$ disjoint open subsets of $\Omega$, say $\Omega^i$ for $1\leq i \leq 2^j$, having $\lambda_1(\Omega^i)\leq K_{j+1}=M'$ for each $i$.\par
To conclude the thesis, it is thus enough to show that
\begin{equation}\label{tofinish}
\lambda_k(\Omega) \leq \lambda_{2^j} (\Omega) \leq \max \Big\{ \lambda_1(\Omega^i):\, 1\leq i \leq 2^j\Big\}\leq M'\,,
\end{equation}
and in fact only the second inequality is to be shown, being the first and the last true by construction.\par
To get~(\ref{tofinish}), for every $1\leq i\leq 2^j$ let $u_i$ be a first eigenfunction of $\Omega^i$, again extended by $0$ on $\Omega\setminus \Omega^i$, so that
\begin{align*}
\int_{\Omega} u_i^2 = \int_{\Omega^i} u_i^2 = 1\,, &&
\int_{\Omega} |Du_i|^2 = \int_{\Omega^i} |Du_i|^2 = \lambda_1(\Omega^i) \leq M'\,,
\end{align*}
and then $\rc(u_i,\Omega) \leq M'$. Observe that the functions $u_i$ are mutually orthogonal (both in the $L^2$ and in the $W^{1,2}$ sense) by construction, since they are supported on disjoint sets. Hence, the linear subspace $K_{2^j}$ of $W^{1,2}_0(\Omega)$ spanned by the functions $u_i$ for $1\leq i \leq 2^j$ is $2^j$-dimensional. Thanks to Theorem~\ref{chareigen}, to prove~(\ref{tofinish}) it is enough to show that $\rc(w)\leq \max\big\{ \lambda_1(\Omega^i):\, 1\leq i \leq 2^j\big\}$ for every $w\in K_{2^j}$. And in fact, writing the generic function $w\in K_{2^j}$ as $w = \sum \beta_i u_i$, by the orthogonality of the different $u_i$ one has clearly
\[\begin{split}
\rc(w,\Omega)&=\frac{\bal\int_\Omega\Big|\sum\beta_iDu_i\Big|^2\eal}{\bal\int_\Omega \Big(\sum\beta_iu_i\Big)^2\eal}
= \frac{\bal \sum \beta_i^2 \int_\Omega \big| Du_i\big|^2\eal}{\bal \sum \beta_i^2 \int_\Omega   u_i^2\eal}
= \frac{\bal \sum \beta_i^2 \,\rc(u_i,\Omega) \int_\Omega   u_i^2\eal}{\bal \sum \beta_i^2 \int_\Omega   u_i^2\eal}\\
&= \frac{\bal \sum \beta_i^2 \,\lambda_1(\Omega^i) \int_\Omega   u_i^2\eal}{\bal \sum \beta_i^2 \int_\Omega   u_i^2\eal}
\leq \max \Big\{ \lambda_1(\Omega^i):\, 1\leq i \leq 2^j\Big\}\,.
\end{split}\]
As noticed before, this gives the validity of~(\ref{tofinish}), hence the proof is concluded.
\end{proof}

To obtain Theorem~\mref{bounded}, we now only need a trivial rescaling argument.

\begin{proof}[Proof of Theorem~\mref{bounded}]
First of all notice that, by density, it is admissible to consider only the case of the open sets. We  apply Lemma~\ref{exthm} with $K=1$, so defining $M:=M'(k,1,N)$. We will prove Theorem~\mref{bounded} with such $M$. Let $\Omega\subseteq \R^N$ be an open set, and apply the rescaling formula~(\ref{rescaling}) choosing $\alpha=\lambda_1(\Omega)^{\frac 12}$, thus getting $\lambda_1(\alpha \,\Omega)=1$. By Lemma~\ref{exthm}, we derive $\lambda_k(\alpha\Omega)\leq M$, and then by~(\ref{rescaling}) again we find $\lambda_k(\Omega) = \alpha^2 \lambda_k (\alpha\,\Omega) \leq M \lambda_1(\Omega)$, thus the proof is concluded.
\end{proof}

\end{document}